\documentclass[11pt]{article}

\usepackage{amsmath}
\usepackage{parskip}
\usepackage{amssymb}
\usepackage{float} 
\usepackage{amsthm}
\usepackage{latexsym}
\usepackage{color}
\usepackage{mathtools}
\usepackage{graphicx}
\usepackage{appendix}
\usepackage{color} 
\usepackage{enumerate}
\usepackage{diagbox}
\usepackage[T1]{fontenc}
\usepackage{palatino}
\usepackage[english]{babel}
\usepackage[table]{xcolor}
\usepackage{calrsfs}
\usepackage{geometry}
\usepackage{url}
\DeclareSymbolFont{calletters}{OMS}{cmsy}{m}{n}
\DeclareSymbolFontAlphabet{\mathcal}{calletters}
\DeclareMathAlphabet{\mathpzc}{OT1}{pzc}{m}{it}

\def\be{\begin{eqnarray}}
\def\ee{\end{eqnarray}}

\def\b*{\begin{eqnarray*}}
\def\e*{\end{eqnarray*}}

\newtheorem{Theorem}{Theorem}[part]
\newtheorem{Definition}{Definition}[part]
\newtheorem{Proposition}{Proposition}[part]

\newtheorem{Assumption}{Assumption}[part]
\newtheorem*{Assumptionn}{Assumption}
\newtheorem{Lemma}{Lemma}[part]

\newtheorem{Remark}{Remark}[part]

\makeatletter \@addtoreset{equation}{section}

\@addtoreset{Definition}{section}

\@addtoreset{Theorem}{section}

\@addtoreset{Proposition}{section}

\@addtoreset{Property}{section}

\@addtoreset{Assumption}{section}

\@addtoreset{Corollary}{section}

\@addtoreset{Lemma}{section}

\@addtoreset{Remark}{section}

\@addtoreset{Example}{section}

\@addtoreset{Condition}{section}

\newcommand{\No}[1]{\left\|#1\right\|}     

\def \R{\mathbb{R}}

 \title{Moral hazard in welfare economics: on the advantage of Planner's advices to manage employees' actions.\footnote{This work is supported by the ANR project Pacman, ANR-16-CE05-0027 and the Chair Financial
Risks (Risk Foundation, sponsored by Société Générale).}}

 \author{By Thibaut Mastrolia\footnote{CMAP--\'Ecole Polytechnique, Route de Saclay 91128 Palaiseau, FRANCE, \texttt{thibaut.mastrolia@polytechnique.edu}} }
 
             \date{\today}

\begin{document}

 \maketitle
 
 \noindent{\footnotesize \textbf{Abstract: }\noindent In this paper, we study moral hazard problems in contract theory by adding an exogenous Planner to manage the actions of many Agents hired by a Principal. We provide conditions ensuring that Pareto optima exist for the Agents using the scalarization method associated with the multi-objective optimization problem and we solve the problem of the Principal by finding optimal remunerations given to the Agents. We illustrate our study with a linear-quadratic model by comparing the results obtained when we add a Planner in the Principal/multi-Agents problem with the results obtained in the classical second-best case. More particularly in this example, we give necessary and sufficient conditions ensuring that Pareto optima are Nash equilibria and we prove that the Principal takes the benefit of the action of the Planner in some cases.}
\vspace{1em}

\vspace{5mm}

\vspace{0.5em}


 \section{Introduction}
 In 1992, Maastricht treaty, considered as the key stage in the European Union construction, proposed the establishment of a single currency for its members to ensure a stability of prices inside the EU.  This date marks the birth of the euro, as the common currency of the eurozone, with the creation of the European Institut Monetary to ensure the introduction of it, which has operated until the creation in 1998 of a central bank for Europe, the European Central Bank. The ECB is the decision-making center of the Eurosystem, which is one of the main component of the EU (see for instance \cite[Chapter 5, Section 2.5]{scheller2004european} for a description of its general role inside the EU) to ensure the economical smooth functioning of the eurozone, and Maastricht Treaty contains some criteria that member states are supposed to respect. For instance, the Article 104-C stipulates that "Member States shall avoid excessive government deficits" and some criteria to respect it concern the Government budget balance, which does not have to exceed 3\% of the GDP and the debt-to-GDP ratio, which does not have to exceed 60\%. If these benchmarks are exceeded, through the Stability and Growth Pact, members of eurozone have accepted to follow a precise procedure each year to reduce their debt. Even before the financial crisis of 2007-2008, these rules have been discussed (see \cite{peet1998maastricht}), and since the Greek government-debt crisis this controversy has been amplified.\footnote{See for instance (see \cite{libe,figaro,telegraph,laufer2008maastricht}).} The fact is that the eurozone is very heterogeneous in the respect of Maastricht criteria as showed in Table \ref{table:dataUE}.
 \begin{table}[H]\label{table:dataUE}\caption{\small \textit{Governement deficit/surplus in \% of GDP and debt-to-GDP ratio in \% for nine EU members in 2015. Data from Eurostat.}}\vspace{0.5em}
 
 \begin{tabular}{|l|l|l|}\hline
EU members & gov. deficit/surplus in \% of GDP& debt-to-GDP ratio in \%\\
&&\\ \hline
Estonia&+0.1&10.1\\
Finland&-2.8&63.6\\
France&-3.5&96.2\\
Germany&+0.7&71.2\\
Greece&-7.5&177.4\\
Ireland&-1.9&78.6\\
Italie&-2.6&132.3\\
Lithuania&-0.2&38.7\\
Spain&-5.1&99.8\\
 \hline Euro area (19 countries)&-2.1&90.4\\
 \hline
\end{tabular} 
 \end{table}
 The natural questions arising are the following:
 \begin{itemize}
 \item[a.] How the ECB can provide incentives to EU members to respect the rules induced by Maastricht treaty? Which kind of procedure is the more efficient?
 \item[b.] How EU members have to interact for the welfare of the European Union? 
 \end{itemize}
 
\noindent The Greek government crisis has impacted all the eurozone and showed that members are strongly correlated each others. Through this example we see that one difficulty of both the ECB and eurozone members is to reach an optimal decision and equilibria to succeed in the global european construction.\vspace{0.5em}

\noindent The example above, and more specially the problem a., describes the kind of investigations made in the incentives theory starting in the 70's (see among others \cite{mirrlees1974notes,mirrlees1976optimal}) and is an illustration of a Principal/Agent problem. More exactly, the classical framework considered is the following: a Principal (\textit{she}) aims at proposing to an Agent (\textit{he}) a contract. The Agent can accept or reject the contract proposed by the Principal and under acceptance conditions, he will provide a work to manage the wealth of the Principal. However, the Principal is potentially imperfectly informed about the actions of the Agent, by observing only the result of his work without any direct access on it. The Principal thus designs a contract which maximizes her own utility by considering this asymmetry of information, given that the utility of the Agent is held to a given level (his reservation utility). From a game theory point of view, this class of problems can be identified with a Stackelberg game between the Principal and the Agents, \textit{i.e.} the Principal anticipates the best-reaction effort of the Agent and takes it into account to maximize her utility. Moral hazard in contracting theory, \textit{i.e} the Principal has no access on the work of her Agent, has been developed during the 80's and was investigated in a particular continuous time framework by Holmstr\"om and Milgrom in \cite{holmstrom1987aggregation}. We refer to the monographies of Laffont and Martimort \cite{laffont2001incentive}, Laffont and Tirole \cite{laffont1996incentives}, Sung \cite{sung2001lectures} and Cvitanic and Zhang \cite{cvitanic2012contract} for nice reviews of the literature in this topic and different situations studied dealing with moral hazard in Principal/Agent problems. More recently, the noticeable work of Sannikov \cite{sannikov2008continuous} investigates a stopping time problem in contract theory by emphasizing the fundamental impact of the value function of the Agent's problem to solve the problem of the Principal. This was then mathematically formalized with the works of Cvitanic, Possama\"i and Touzi in \cite{cvitanic2014moral,cvitanic2017dynamic} by proposing a nice handleable procedure to solve a large panel of problems in moral hazard. More exactly, they have showed that the Stackelberg equilibrium between the Agent and the Principal may be reduced to two steps in the studying of general problems in contracts theory. First, given a fixed contract proposes by the Principal, the Agent aims at maximizing his utility by finding the best reaction effort associated with the proposed contract. It is well-known, since the works of Rouge and El Karoui \cite{rouge2000pricing} and Hu, Imkeller and M\"uller \cite{hu2005utility} that a utility maximization problem can be reduced to solve Backward Stochastic Differential Equations (BSDE for short, we refer to the works of Pardoux and Peng \cite{pardoux1990adapted,pardoux1992backward} and El Karoui, Peng and Quenez \cite{el1997backward} for general results related to this theory), and Cvitanic, Possama\"i and Touzi have proved that more generally, when the Agent can control the volatility, the problem can be reduced to solve a second order BSDE (see for instance the seminal work \cite{soner2012wellposedness} and the extension of it with more general conditions \cite{possamai2015stochastic}). Then, it is proved in \cite{cvitanic2014moral,cvitanic2017dynamic} that the problem of the Principal can be reduced to solve a stochastic control problem with state variables of the problem the output and the value function of the Agent, by using the HJB equations associated with it and verification theorems.\vspace{0.3em}

\noindent An extension of moral hazard with a Principal and an Agent, which echoes the example of the ECB as the Principal and the EU members as the Agents presented above, consists in considering a Principal dealing with many Agents. Principal/Multi-Agents problems have been investigated in a one period framework by H\"olmstrom \cite{holmstrom1982moral} (among others) and then extended in the continuous time model by Koo, Shim and Sung in \cite{keun2008optimal} and more recently by Elie and Possama\"i in \cite{elie2016contracting}. In the latter, Elie and Possama\"i consider exponential utility functions for the Agents and the Principal and they assume that the Agents are rational so that the first step of the procedure to solve the Agents' problems remains to find Nash equilibria, which can be reduced to solve a multi-dimensional quadratic BSDE, as explained in \cite{elie2016contracting}. Nevertheless, by recalling the example of the ECB and EU members above with question b., we can also consider an other type of interactions between the Agents.
\vspace{0.5em}

 \noindent In microeconomics, we can distinguish two type of interactions between connected agents. Agents can be considered as rational economical entities and aim at finding their best reaction functions to maximize their wealths/minimize their costs, in view of the actions of the others. This investigation consists in finding a Nash equilibrium and fits with a situation in which the (non-cooperative) agents cannot deviate from this equilibrium. However, as explained with the so-called prisoner's dilemma (see Table 2 for more explanation), the configuration obtained is not necessarily an optimal choice for the welfare of the system.

\begin{table}[H]\caption{\small \textit{This example was emphasized by Merrill Flood and Melvin Dresher in the 50's and then formalized by Albert W. Tucker. Considers two prisoners P1 and P2 waiting for their interrogation to determine if they are indeed guilty or not. The prisoners have two choices, they can keep silent or they can denounce the other prisoner. We sum up the payoffs-vector associated with this situation in the table below, the first (resp. second) component of the vector is the payoff of P1 (resp. P2). We see that Pareto optima are keep silent-keep silent, denounce-keep silent and keep silent-denounce and the Nash equilibrium is denounce-denounce, which is the dominant strategy. In particular, if the two prisoners are forced to play a Pareto optimum, their global payoff is $0$ or $-9$ which is always better than the natural equilibrium with global payoff $-10$.}}\vspace{1em}
 
\centering \begin{tabular}{|l|l|l|}\hline
\backslashbox{Action of P2}{Action of P1} & keep silent & denounce\\
 \hline keep silent & $(0,0)$& $(1,-10)$\\
 \hline denounce & $(-10,1)$ &$(-5,-5)$\\
 \hline
\end{tabular} 
 \end{table} 
 
\noindent Welfare economics is a part of microeconomics which aims at determining the degree of well-being of a situation and Pareto optima are considered as one criterion to measure it. Indeed, a Pareto optimum consists in finding a configuration in which all the considered entities cannot deviate without harming the state of an other entity. Unlike a Nash equilibrium, a Pareto optimum is not reached using dominant strategies but has to be imposed in a general situation, since rational entities will reasonably converge to a Nash equilibrium. Finding Pareto optima gives a lot of information in terms of general equilibrium inside a system of markets and supply/demand systems, as showed by (among others) Walras in \cite{walras1896elements} to explain price-setting mechanisms, and mathematically formulated by Arrow and Debreu in \cite{arrow1954existence} with the celebrated two fundamental theorems of welfare economics. Roughly speaking, the first fundamental theorem of welfare economics states that any general equilibrium in a competitive market is (weakly) Pareto optimal. The second states that any Pareto optimum can lead to a general equilibrium by reallocating initial appropriations. We refer to \cite{mirrlees1974notes, arnott1991welfare} and the monography \cite{mas1995microeconomic}  for more details on it. It is however\footnote{We for instance refer to \cite{Acemoglu} for an investigation of this kind of issues.} not clear that with information asymmetry this equivalence is always true (specially the second fundamental theorem does not hold), but studying the existence of Pareto optima seems to be interesting to have relevant information related to general equilibrium in view of the first theorem. Moral hazard problems dealing with Pareto optimality was studied in a one-period model in \cite{prescott1984pareto} and in a two-period model in \cite{panaccione2007pareto}. As far as we know, it does not have been investigated in continuous time models and the present paper is the first who considers it. Before going further, let us explain how we have to understand the Principal/multi-Agents problem studied in this work. Since the actions of Agents lead naturally to a Nash equilibrium (which does not coincides \textit{a priori} with a Pareto optimum), we have to assume that the Agents cannot manage their work themselves or are forced to follow a precise strategy. This induces to introduce a third player in the Principal/multi-Agents game, namely \textit{the Planner}, who imposes an effort to the Agents for their well-being. More specifically, the Planner can be seen as a mediator inside a firm who managed the actions of agents or a regulator who forces the Agents to act for the global interest. The Planer can be for instance a Government who imposes some Labour laws, by thinking about the global interest of the employees or any other entity who manages actions of Agents hired by a Principal. In this paper, we will distinguish the case where no Planner impacts the Stackelberg equilibrium, which coincides with the second best case in contract theory and where rational Agents reach to Nash equilibria, to the case in which a Planner manages the work of the Agents. The structure of the paper is the following:
\vspace{0.3em}

\noindent After having described the Economy studied in Section \ref{section:economy}, we extend the result of \cite{elie2016contracting} to general separable utilities for the Agents, and general utilities for the Principal in Section \ref{section:noplanner} for the classical second-best case, named the \textit{no-Planner model}. We provide conditions on the data ensuring that the multi-dimensional BSDE associated with Nash equilibria for the Agents is well-posed as an application of the recent results in \cite{harter2016stability} (see Appendix \ref{appendix:multidimBSDE} and Remark \ref{remark:qgbsde}). Then, in Section \ref{section:MOOP} we turn to the model in which a Planner manages the action of Agents. Using the fact that finding Pareto optima can be reduced to solve a Multi-Objective Optimization Problem (MOOP for short), we give general forms of some Pareto optima through the solutions of BSDEs and we solve the problem of the Principal using the HJB equation associated with it. Finally, in Section \ref{section:application}, we apply our results to a linear-quadratic model similar to the applied model studied in \cite{elie2016contracting} with two Agents having appetence coefficients and we give some interpretations on the relevant results. We compare the no-Planner model with models in which a Planner intervenes in the Principal/multi-Agents problem by providing sufficient and necessary conditions such that a Nash equilibrium is Pareto efficient, and by comparing the value functions of the Principal. \vspace{0.5em}

\noindent The notations and all the technical details are postponed to Appendix \ref{appendix:model} and \ref{appendix:multidimBSDE} and proofs of Section \ref{section:MOOP} are postponed to Appendix \ref{appendix:proof} to allege the reading of the paper. 
\section{The Economy}\label{section:economy}

We consider a finite number $N\geq 1$ of Agents hired by one Principal. Each Agent receives from the Principal a salary at a fixed horizon $T>0$ to manage the project of the Principal described by an $\mathbb R^N-$valued process $X$ with volatility an $\mathcal M_{N}(\mathbb R)-$valued map denoted by $\Sigma$ depending on the output $X$, given by
\begin{equation}\label{dynamic:output}
X_t:= \int_0^t \Sigma(s,X_s) dW_s,\; t\in [0,T],\, \mathbb P-a.s., 
\end{equation}
where $W$ is an $N-$dimensional Brownian motion defined on some probability space $(\Omega,\mathcal F,\mathbb P)$ filtred with the natural filtration of the Brownian motion $\mathbb F:=(\mathcal F_t)_{t\in[0,T]}$. We assume in this paper that the volatility $\Sigma$ is uniformly bounded, the map $(t,x)\longmapsto \Sigma(t,X_t)$ is $\mathbb F-$progressively measurable, invertible and such that \eqref{dynamic:output} has a unique strong solution with exponential moments.\vspace{0.3em}

\begin{Remark}
The boundedness of $\Sigma$ is here for the sake of simplicity, while the invertibility is fundamental to define properly the weak formulation of the problem as explained in \cite{elie2016contracting}.
\end{Remark}

\subsection{Impact of the actions of $N-$interacting agent}
Let $A$ be a subset of $\mathbb R$. We assume that any agent can impact, with his effort, both the value of his assigned project and the values of the other projects. More explicitly, any $i$th Agent has an impact on the $i$th component on $X$ and on the other components managed by the other Agents. We represent the general action of the $N$ Agents by a matrix $a\in \mathcal M_N(A)$ where $a^{j,i}$ is the action of the $i$th Agent on the project $j$. Thus, $a^{:,i}$ is the $A^N-$valued column vector of the $i$th Agent's actions. This matrix impacts\footnote{See Appendix \ref{appendix:model} for mathematical details.} the dynamic \eqref{dynamic:output} of $X$ by adding a drift $b$ defined as a map from $[0,T]\times \Omega\times \mathcal M_{N}(A)$ such that for any time $t$, any space variable $x\in \mathbb R^N$ and any effort $a\in \mathcal M_N(A)$ of the Agents, we have
$$b(t,x,a):=\left(b^i\left(t,x,(a^{i,1}, \dots, a^{i,N})^\top\right)\right)_{1\leq i\leq N}. $$

\noindent For technical reasons, we have to focus on \textit{admissible} actions of the Agents as a subspace $\mathcal A$ of $\mathcal M_N(A)-$valued process for which the model is well-posed. We refer to Appendix \ref{appendix:model} for a mathematical definition of it. Thus, under technical conditions given in the appendix, the dynamic of $X$, impacted by any admissible actions $a$ of Agents, is given by
\begin{equation}\label{dynamicX}
X_t= \int_0^t b(s,X_s, a_s) ds+\int_0^t \Sigma(s,X_s) dW_s^a,
\end{equation}
where 
$$W^a:= W-\int_0^T \Sigma(s,X_s)^{-1} b(s,X_s,a_s) ds,$$ is a Brownian motion under some probability measure $\mathbb P^a$ equivalent to $\mathbb P$. We thus work under the weak formulation of the problem (see \cite[Sections 2.1 and 2.2]{elie2016contracting} for more details on it).
\paragraph{Interpretation of \eqref{dynamicX}.} We first focus on the dependance with respect to $X$ in both the volatility $\Sigma$ and the drift $b$ of the project. This classical phenomenon expresses some friction effects between the different projects managed by the Agents so that each component of the general outcome $X$ can depend on the values of the others. Turn now to the dependance with respect to the effort of the Agents. The dynamic of the project $i$ depends on all the efforts provided by the $N$ Agents and assuming that the $i$th Agent has preferences depending on the $i$th component of $X$, this model emphasizes exactly an interacting effect between Agents since each Agent can impact positively or negatively the project managed by an other.

\subsection{Stackelberg equilibrium, Nash equilibrium and Pareto optimality}

It is well known that a moral hazard problem can be identified with a Stackelberg equilibrium between the Principal and the Agents. In the considered system, the first step consists in finding a best respond effort provided by the Agents given a sequence of salaries proposed by the Principal. In the second step, taking into account the reactions of the Agents, the Principal aims at maximizing her own utility.\vspace{0.3em}

\noindent We assume that each Agent is penalized through a cost function depending on his effort. More precisely, we denote by $k^i:[0,T]\times \mathbb R^N\times A^N\longrightarrow \mathbb R$ the cost function associated with the $i$th Agent.\vspace{0.3em}

\noindent For any $i\in \{1,\dots, N \}$, we set $U_0^i:\mathbb R \times \mathcal M_N(A)\longrightarrow \mathbb R$ the utility of the $i$th Agent. We will consider in this paper separable utilities given a sequence $(\xi^i)_{1\leq i\leq N}$ of salaries given by the Principal and an admissible action $a$ chosen by the $N$-Agents, such that 
$$U_0^i(\xi^i,a):=\mathbb E^{\mathbb P^a}\left[U_A^i(\xi^i)+\Gamma_i(X_T) - \int_0^T k^i(t,X_t, a^{:,i}) dt \right],$$ where $U_A^i:\R\longrightarrow\mathbb R$ is a concave and increasing map so that its inverse $U_A^{(-1)}$ is well-defined and increasing, $\Gamma_i:\mathbb R^N\longrightarrow \mathbb R$ is the appetence function of Agent $i$ to manage project $i$ which satisfies a technical, but not too restrictive, assumption (see Assumption $\textbf{(G)}$ in Appendix \ref{appendix:model}). For instance, $\Gamma_i(x)=\gamma_i\left( x^i-\frac{1}{N-1}\sum_{j\neq i} x^j\right)$, where $\gamma_i$ is the appetence of Agent $i$ toward the project $i$ compared to the other projects.\vspace{0.3em}

\noindent We recall that any Agent can accept or reject the contract proposed by the Principal through a reservation utility constrain. More exactly, we denote by $\mathcal C$ a general set of \textit{admissible}\footnote{See Appendix \ref{appendix:model} for mathematical details with the definition of $\mathcal C$.} contracts proposed by the Principal to the Agents, which remains to put technical assumptions ensuring that all the mathematical objects used in this paper are well-defined and such that for any panel of contracts $\xi:=(\xi^i)_{1\leq i\leq N}$ the following reservation utility constrains are satisfied
\begin{equation}\label{reservation:constrain}\sup_{a\in\mathcal A} U_0^i(\xi^i,a)\geq R_0^i, \; 1\leq i\leq N,\end{equation}
with a reservation utility $R_0^i\in \mathbb R$ for the $i$th Agent. We denote $R_0:=(R_0^i)_{1\leq i\leq N}$ the vector of reservation utilities.\vspace{0.3em}

\noindent  An interesting problem concerns the intrinsic type of best reaction efforts provided by the Agents. Two different approaches can be investigated.\vspace{0.3em}

\noindent One can assume that each Agent provides an optimal effort in view of the actions of the others. In this case, each Agent aims at finding a best reaction effort given both a salary proposed by the Principal and other Agents efforts. This situation typically fits with a non-cooperative game between Agents which reach to find stable equilibria of type Nash equilibria. This problem was well investigated in \cite{elie2016contracting} by proving that Agents play an equilibrium in view of performance of others players. Mathematically\footnote{See paragraphs Notations and General model and definitions in Appendix \ref{appendix:model} for the definitions of the operator $\otimes$ and the space of admissible best reaction $\mathcal A^i$ with \eqref{def:calAi}.}, we define a Nash equilibrium for the Agents by the following \vspace{0.5em}

\begin{Definition}[Nash equilibrium] Given an admissible contract $\xi\in \mathcal C$, a Nash equilibrium for the $N$ Agents is an admissible effort $e^\star(\xi)\in \mathcal A$ such that for any $i\in\{1,\dots,N\}$ \begin{equation}\label{nash:def}\sup_{e\in \mathcal A^i((e^\star(\xi))^{:,-i})} \, U_0^i(\xi^i, a\otimes_i (e^\star(\xi))^{:,-i}) = U_0^i(\xi^i, (e^\star(\xi))^{:,i}\otimes_i (e^\star(\xi))^{:,-i}).\end{equation}
\end{Definition}

\noindent The main problem of Nash equilibrium is that it is not optimal in general, from a welfare economics point of view, of interacting entities (see the prisoner's dilema in Table 2). It is why in this paper we will mainly focus on Pareto optima, mathematically defined by  \vspace{0.5em}
\begin{Definition}[Pareto optimum]
An admissible action $a^*\in \mathcal A$ is Pareto optimal if it does not exist $a\in \mathcal A$ such that $U_0^i(\xi,a^*) \leq U_0^i(\xi,a)$ for any $i=1,\dots, N$ and $U_0^i(\xi,a^*)<U_0^i(\xi,a)$ for at least one index $i$. 
\end{Definition}

\noindent The studied model coincides with a Principal/multi-Agents problem in which an exogenous entity, called the Planner, forces the Agents to act for the global interest. Given an optimal remuneration, the Planner aims at maximizing the global utility of the $N$-Agents by finding a Pareto equilibrium. Then, taking the best reaction effort of the Agents into account, the Principal chooses among the set of contracts those who maximize her utility.\vspace{0.3em}

\section{Some reminders on the no-Planner model}\label{section:noplanner}
We assume that no Planner intervenes in the Stackelberg games between Agents and the Principal and we extend merely the main results in \cite{elie2016contracting} to the case of separable utilities and we omit the proof since they follow exactly the same lines that \cite{elie2016contracting}. In that case, we consider rational Agents reaching a Nash equilibrium given an admissible sequence of contracts $(\xi^i)_{1\leq i\leq N}$, which corresponds to the classical second-best case. Intuitively, we begin to set the paradigm of any $i$th Agent. Given a salary $\xi^i$ given by the Principal and an effort $a^{:,-i}$ provided by the others, Agent $i$ aims at solving
\begin{equation}\label{pb:agenti}
U_0^i(\xi,a^{:,-i}):=\sup_{a\in \mathcal A^i(a^{:,-i})}\, \mathbb E^{\mathbb P^{a\otimes_i a^{:,-i}}} \left[U_A^i(\xi^i)+\Gamma_i(X_T) - \int_0^T k^i(t,X_t, a^{:,i}) dt \right].
\end{equation} 
Following the same computations than those in \cite{elie2016contracting}, by using martingale representation theorems, Problem \eqref{pb:agenti} remains to solve
\begin{equation}\label{bsde:i:intuition}
\mathcal Y_t^i=U_A^i(\xi^i)+\Gamma_i(X_T)+\int_t^T \sup_{a \in \mathcal A^i(a^{:,-i})}\left\{b(s,X_s, a\otimes_i a^{:,-i})\cdot \mathcal Z_s^i- k^i(s,X_s,a) \right\}ds-\int_t^T \mathcal Z_s^i\cdot \Sigma_sdW_s,
\end{equation}
and by denoting $a_{NA}^\star$ the maximizer of the generator of this BSDE, we deduce that the $A^N$-valued process $a_{NA}^\star(\cdot, X_\cdot, \mathcal Z_\cdot^i, a^{:,-i})$ is the best reaction effort of Agent $i$ given a salary $\xi^i$ and effort of the other players $a^{:,-i}$, where $(\mathcal Y^i, \mathcal Z^i)$ is the unique solution to BSDE \eqref{bsde:i:intuition}, under technical assumption stated in Appendix \ref{appendix:model}. \vspace{0.5em}

\subsection{Nash equilibrium and multidimensional BSDE}

Now, each Agent computes his best reaction effort at the same time. We thus have to assume that for any $(t,z,x)\in [0,T]\times \mathcal M_N(\mathbb R)\times \mathbb R^N$ there exists a fixed point $a_{NA}^\star(t,z,x)\in \mathcal M_N(A)$ inspiring by the best reaction effort of Agent $i$ made before. We consider that the following assumption holds in the following
\begin{Assumption} For any $(t,z,x)\in [0,T]\times \mathcal M_N(\mathbb R)\times \mathbb R^N$, there exists $a_{NA}^\star(t,z,x)\in\mathcal M_N(A)$ such that
$$(a^\star_{NA})^{:,i}(t,z,x) \in \underset{a\in A^N}{\text{ arg max }} \left\{\sum_{j=1}^N b^j\left(t,x, (a\otimes_i (a^\star_{NA})^{:,-i}(t,z,x))^{j,:} \right)z^{j,i}- k^i(t,x,a) \right\}.$$ 
We denote by $\mathcal A^\star_{NA}(t,z,x)$ the set of fixed point $a_{NA}^\star(t,z,x)$.
\end{Assumption}
\noindent As emphasized by \cite{elie2016contracting}, finding a Nash equilibrium for any sequence of salaries $(\xi^i)_{1\leq i\leq N}$ is strongly linked to find a solution to the following multidimensional BSDE 

\begin{equation}\label{edsr:max:nash}
\mathcal Y_t=(U_A^i(\xi^i))_{1\leq i\leq N}+ (\Gamma_i(X_T))_{1\leq i\leq N}+\int_t^T f^\star_{NA}(s,X_s,\mathcal Z_s)ds-\int_t^T \mathcal Z_s^\top \Sigma_sdW_s,
\end{equation}
with 
$$f^{\star,i}_{NA}(t,x,z):= \sum_{j=1}^N b^j(t,x, (a^\star_{NA})^{j,:})z^{ji}- k^i(t,x,(a^\star_{NA})^{:,i}),\; 1\leq i\leq N. $$
We recall the following definition of a solution to BSDE \eqref{edsr:max:nash} \vspace{0.5em}

\begin{Definition}\label{def:sol:multidimBSDE} A solution of BSDE \eqref{edsr:max:nash} is a pair $(Y,Z)\in \mathcal S^\infty(\mathbb R^N)\times (\mathcal S^\infty(\mathbb R^N)\times \mathcal H^m_{\text{BMO}}(\mathbb R^N))$ for some $m>1$ satisfying Relation \eqref{edsr:max:nash}.
\end{Definition}

\noindent Similarly to Theorem 4.1 in \cite{elie2016contracting}, we can now state that there is a one to one correspondance between a Nash equilibrium and a solution to BSDE \eqref{edsr:max:nash}. The proof of this result follows exactly the same lines that the proof of Theorem 4.1 in \cite{elie2016contracting}.\vspace{0.5em}

\begin{Theorem}\label{thm:multiqgbsde}
There exists a Nash equilibrium $e^\star(\xi)$ if and only if there is a unique solution of BSDE \eqref{edsr:max:nash} in the sense of Definition \ref{def:sol:multidimBSDE}. In this case, we have
$e_t^\star(\xi)\in \mathcal A^\star_{NA}(t,X_t,\mathcal Z_t)$ and conversely, any $a_{NA}\in \mathcal A^\star_{NA}(t,X_t,\mathcal Z_t)$ is a Nash equilibrium.
\end{Theorem}\vspace{0.5em}

\begin{Remark}\label{remark:qgbsde}
The previous theorem emphasizes that the existence of a Nash equilibrium is connected to the existence of a unique solution in the sense of Definition \ref{def:sol:multidimBSDE}. As explained in \cite{frei2011financial}, for instance, the problem is mathematically ill-posed in general and \cite{elie2016contracting} circumvents this problem by imposing the existence of a solution to BSDE \eqref{def:sol:multidimBSDE} in the definition of admissible contracts. However, as soon as the class of admissible contracts is sufficiently \textit{regular}, recent results \cite{xing2016class,harter2016stability} can be applied to ensure that the multidimensional BSDE \eqref{edsr:max:nash} is well-posed. We refer to Appendix \ref{appendix:multidimBSDE} for more details on this class of admissible contracts and Proposition \ref{prop:existence:qgbsde} which gives conditions ensuring that their exists a unique solution in the sense of Definition \ref{def:sol:multidimBSDE} of this BSDE.
\end{Remark}
\subsection{The problem of the Principal and multidimensional HJB equation}
As previously, this section is very informal to allege the reading since it is a mere extension of \cite{elie2016contracting} to separable and general utilities. The results are again completely similar to those in \cite{elie2016contracting} and can be proved by following exactly the same lines. For the sake of readability, we prefer to omit it to focus on the proofs of Section \ref{section:MOOP} below, which is the real contribution of this paper. \vspace{0.5em}

\noindent Assume that BSDE \eqref{edsr:max:nash} admits a unique solution $(Y,Z)\in \mathcal S^\infty(\mathbb R^N)\times (\mathcal S^\infty(\mathbb R^N)\times \mathcal H^m_{\text{BMO}}(\mathbb R^N))$ for some $m>1$. Let $a^\star_{NA}$ be a Nash equilibrium selected by the Agents (see \cite[Section 4.1.4]{elie2016contracting} for some selection criterion) or the Principal\footnote{In fact, if the Agents cannot select a Nash equilibrium, then the problem of the Principal is $$U_0^{P,NA}=\underset{a^\star_{NA}\in \mathfrak A^\star_{NA}(X,\mathcal Z)}{\sup_{\xi=(\xi^i)_{1\leq i\leq N} \in \mathcal C}} \mathbb E^{\mathbb P^{a^\star_{NA}(\cdot,X, \mathcal Z)}} \left[\mathcal U_P\left( \ell(X_T) - \sum_{i=1}^N \xi^i\right) \right], $$ where $ \mathfrak A^\star_{NA}$ is the restriction of $\mathcal A^\star$ with selection criterion for the Agents, and the results below are completely similar. It is why for the sake of simplicity, we assume that only one Nash equilibrium is selected here, to allege the reading.} if the Agents cannot select it. Recall that the Principal aims at solving
\begin{equation}\label{pb:primal:principal:nash}
U_0^{P,NA}=\sup_{\xi=(\xi^i)_{1\leq i\leq N} \in \mathcal C} \mathbb E^{\mathbb P^{a^\star_{NA}(\cdot,X, \mathcal Z)}} \left[\mathcal U_P\left( \ell(X_T) - \sum_{i=1}^N \xi^i\right) \right].
\end{equation}
As explained in \cite{sannikov2008continuous, cvitanic2014moral,cvitanic2017dynamic}, by mimicking \cite[Section 4.2]{elie2016contracting} and in view of the decomposition \eqref{edsr:max:nash} of admissible contracts, one can show\footnote{See Section \ref{section:characterization} which provides a path to do it or \cite{elie2016contracting} for exponential utilities.} that solving this problem remains to solve a stochastic control problem with two state variables: the value of the firm $X$ and the value function of the $i$th Agent for any $i\in \{1,\dots,N\}$. This suggests to introduce the following Hamiltonian $H^{NA}:[0,T]\times \mathbb R^N\times \mathbb R^N\times  \mathbb R^N\times \mathcal M_N(\mathbb R)\times \mathcal M_N(\mathbb R) \times \mathcal M_N(\mathbb R)\longrightarrow \mathbb R$ by
\begin{align*}
H^{NA}(t,x,p,q,P,Q,R)&:=\sup_{z\in \mathcal M_N(\mathbb R)}\Big\{ b(t,x, a^\star_{NA}(t,x,z))\cdot p+  f^{\star}_{NA}(t,x,z)\cdot q \\
&+\frac12 \text{Tr} \left( \Sigma(t,x) \Sigma(t,x)^\top P\right)+\frac12 \text{Tr} \left(z^\top \Sigma(t,x) \Sigma(t,x)^\top z Q\right)\\
&+\text{Tr}\left( \Sigma(t,x) \Sigma(t,x)^\top z R\right) \Big\}.
\end{align*}

\noindent Thus, the HJB equation associated with \eqref{pb:primal:principal:nash} is 
$$\textbf{(HJB-Na)}\begin{cases}
\displaystyle -\partial_t v(t,x,y)-H^{NA}(t,x,\nabla_x v,\nabla_y v,\Delta_{xx}v,\Delta_{xy}v,\Delta_{xy}v)&=0,\; (t,x,y)\in [0,T)\times \R^N\times \R^N\\
v(T,x,y)=\mathcal U_P\left(\ell(x)-\sum_{i=1}^N (U_A^{i})^{(-1)}(y^i-\Gamma_i(x))\right).
\end{cases}$$
As usual, using a classical verification theorem (see for instance \cite{touzi2012optimal}) we get the following result\vspace{0.5em}

\begin{Theorem} Assume that PDE \textbf{(HJB-Na)} admits a unique solution $v$ continuously differentiable with respect to $t$ and twice continuously differentiable with respect to its spaces variables and that for any $t,x,p,q,P,Q,R\in [0,T]\times \mathbb R^N\times \mathbb R^N\times  \mathbb R^N\times \mathcal M_N(\mathbb R)\times \mathcal M_N(\mathbb R) \times \mathcal M_N(\mathbb R)$, the supremum in $H^{NA}(t,x,p,q,P,Q,R)$ is attained for (at least) one $z^\star_{NA}(t,x,p,q,P,Q,R)$. Moreover, for any $Y^i_0\geq R_0^i, \; 1\leq i\leq N$, we assume that the following coupled SDE
$$\begin{cases}
\displaystyle X^\star_t&=x+\int_0^t \Sigma(s,X_s^\star)dW_s,\\
\displaystyle Y_t^{\star,Y_0}&=Y_0-\int_0^t  f^\star_{NA}(s,X^\star_s,z^{\star,NA}_s) ds+\int_0^t (z^{\star,NA}_s)^{\top}\Sigma(s,X^\star_s)dW_s,
\end{cases}$$
admits a unique solution $(X^\star,Y^{\star,Y_0})$ where 
$$z^{\star, NA}_t:=z^\star_{NA}(t,X^\star_t,\nabla_x v,\nabla_y v,\Delta_{xx}v,\Delta_{yy}v,\Delta_{xy}v).$$
If moreover $z^{\star,NA}\in \mathcal H^m_{\text{BMO}}(\mathcal M_N( \R))$ and $(U_A^i)^{(-1)}((Y_T^{\star, R_0})^i-\Gamma_i(X^\star_T))\in \mathcal C$ for any $1\leq i\leq N$, then $$\xi:=\left((U_A^i)^{(-1)}((Y_T^{\star,R_0})^i-\Gamma_i(X^\star_T))\right)_{1\leq i\leq N}$$ is an optimal contract which solves the Principal problem \eqref{pb:primal:principal:nash} with
$$U_0^{P}=v(0,x,R_0). $$
 \end{Theorem}

\section{Multi Objective Optimization Problem with separable preferences}\label{section:MOOP}
Assume now that the Agent does not manage their efforts which are chosen by an exogenous Planner who acts for the welfare of the Agents. We fix a sequence of admissible contracts $\xi:=(\xi^i)_{1\leq i\leq N}$. As before, in all this section, we work under technical assumptions on the coefficients $b,k$ given in Appendix \ref{appendix:model} (see Assumption \ref{assumption:bk}).

\subsection{Solving the multi objective problem using scalarization}
Let $(\lambda_i)_{1\leq i\leq N}$ be a sequence of nonnegative reals such that $\sum_{i=1}^N \lambda_i=1$ and consider the multi-objective optimization problem
\begin{equation}\label{pb:multiojective:weight}
\sup_{a\in \mathcal A} \sum_{i=1}^N \lambda_i \mathbb E^{\mathbb P^a}\left[U_A^i(\xi^i)+ \Gamma_i(X_T)-\int_0^T k^i(t,X_t,a^{:,i}_t) dt \right].
\end{equation}

\noindent The following proposition gives sufficient conditions to find Pareto optima through solutions to the MOOP.  \vspace{0.5em}

\begin{Proposition}[Theorem 3.1.2 in \cite{miettinen2012nonlinear}]\label{prop:pareto}
If their exists $\lambda_i>0$ for any $i\in \{1,\dots,N\}$ with $\sum_{i=1}^N \lambda_i=1$ such that the multi-objective optimization problem \eqref{pb:multiojective:weight} has a solution denoted by $a^\star(\xi, \lambda_1,\dots, \lambda_N)$, then $a^\star(\xi, \lambda_1,\dots, \lambda_N)$ is Pareto optimal.
\end{Proposition}
\noindent The coefficient $\lambda_i$ can be seen as the part chosen by the Planner of the utility of Agent $i$ to maximize the general weighted utility of the Agents.\footnote{See \cite[Proposition 16.F.1]{mas1995microeconomic}, $\lambda^i$ can also be seen as the inverse of the marginal utility of Agent $i$. }
\noindent Fix a sequence $\lambda:=(\lambda_i)_{1\leq i\leq N}$ with $\lambda_i>0, \, i=1,\dots, N$. We set for any $(t,x,a)\in [0,T]\times \mathbb R^N\times \mathcal M_N(A)$
$$U_A^\lambda(\xi):= \sum_{i=1}^N \lambda_i U_A^i(\xi^i), \; \Gamma^\lambda(x):=\sum_{i=1}^N\lambda_i\Gamma_i(x),\; k^\lambda(t, x, a):= \sum_{i=1}^N\lambda_i k^i(t,x,a^{:,i}),
$$
Problem \eqref{pb:multiojective:weight} thus becomes
\begin{equation}\label{pb:multiojective:weight:bis}
u_0^\lambda(\xi):=\sup_{a\in \mathcal A} u_0^\lambda(\xi,a),
\end{equation}
with
$$ u_t^\lambda(\xi,a):= \mathbb E^{\mathbb P^a}\left[U_A^\lambda(\xi)+\Gamma^\lambda(X_T)-\int_t^T k^\lambda(t,X_t,a_t) dt \Big| \mathcal F_t\right],\, t\in [0,T].$$

\noindent We consider the following BSDE for any $a\in \mathcal A$
\begin{equation}\label{bsde:agent:a}
Y_t^{\lambda, a}= U_A^\lambda(\xi)+\Gamma^\lambda(X_T) +\int_t^T \left(b(s, X_s, a_s)\cdot Z_s^{\lambda, a} - k^\lambda(s,X_s, a_s)\right)ds-\int_t^T Z_s^{\lambda, a} \cdot dX_s.
\end{equation}

\noindent We thus have the following Lemma, whose the proof is postponed to Appendix \ref{appendix:proof}\vspace{0.5em}

\begin{Lemma}\label{lemma:bsdeagent} BSDE \eqref{bsde:agent:a} admits a unique solution $(Y^{\lambda, a}, Z^{\lambda,a})\in \mathcal E(\R)\times \bigcap_{p\geq 1} \mathcal H^p(\mathbb R^N)$ such that
$$Y_t^{\lambda, a}=u_t^\lambda(\xi,a). $$
\end{Lemma}
\noindent Let $\mathcal A^\star(x,z,\lambda)$ be defined for any $(x,z,\lambda)\in \mathbb R^N\times\mathbb R^N\times (0,1)^N$ by
$$ \mathcal A^\star(x,z,\lambda):=\left\{ \alpha\in \mathcal A,\; \alpha_s\in \underset{a\in \mathcal M_N(A)}{\text{argmax}}\left\{ b(s, x, a)\cdot z - k^\lambda(s,x, a)\right\} ,\; \text{for a.e. } s\in [0,T]\right\}.  $$
We define for any $a^\star(x,z,\lambda)\in \mathcal A^\star(x,z,\lambda)$
$$f^\star(s,x,z,\lambda):= b(s, x, a^\star(x,z,\lambda))\cdot z - k^\lambda(s,x, a^\star(x,z,\lambda)). $$ We consider the following BSDE
\begin{equation}\label{BSDE:agent:sol}
Y_t^\lambda=U_A^\lambda(\xi)+\Gamma^\lambda(X_T)+\int_t^T f^\star(s,X_s,Z_s^\lambda,\lambda) ds-\int_t^T Z_s^\lambda\cdot dX_s
\end{equation}
The following result solves the problem of the Planner and we refer to Appendix \ref{appendix:proof} for its proof.\vspace{0.5em}
\begin{Theorem}\label{thm:pbPlanner} BSDE \eqref{BSDE:agent:sol} admits a (unique) solution denoted by $(Y^\lambda, Z^\lambda)\in \mathcal E(\mathbb R)\times \bigcap_{p\geq 1}\mathcal H^p(\mathbb R^N)$ such that
$$Y_0^\lambda=u_0^\lambda(\xi),$$
and any $a^\star$ in $\mathcal A^\star(X,Z^\lambda,\lambda)$ is Pareto optimal for any $\lambda\in (0,1)^N$ such that $\lambda\cdot \mathbf 1_N=1$.
\end{Theorem}\vspace{0.5em}

\begin{Remark}\label{remak:choix}
Concerning the choice of the parameter $\lambda$ by the Planner, one can assume that the Planner is penalized by a bad choice of Pareto optima given some external criteria. For instance, the Planner could be forced to choose a class of $\lambda$ depending on the performance of any Agents relatively to the others, \textit{e.g.} $\lambda_i\in [\frac{k^{ii}}{\sum_{j=1}^N k^{ji}}\pm \varepsilon_i]$ for some $\varepsilon_i>0$ with $k^{ji}>0$ the underlying cost of the Agent $i$ to manage project $j$. Mathematically, the Planner aims at solving for instance
$$\min_{\lambda\in (0,1)^N}\mathbb E^{\mathbb P^{a^\star(X,Z^\lambda,\lambda)}}\left[ \int_0^T c^{Pe}(s,X_s,\lambda) ds\right],$$ for some cost function $c^{Pe}$ depending on the characteristics of the Agents. An other section criterion could be to take any parameter $\lambda$ and the corresponding Pareto optima which are Nash equilibria. Thus, we deal in the following with a parameter $\lambda^\star$ chosen by the Planner due to some selection criterions. If this $\lambda^\star$ is not unique, we assume that the Principal have the final say on this choice and maximizes her utility over all the selected Pareto optima. We also assume that in this case that $\mathcal A^\star(X,Z^{\lambda^\star},\lambda^\star)$ is reduced to a singleton for the sake of simplicity (if not and as usual, the Principal also maximizes her utility on any Pareto optima with parameter $\lambda^\star$).
\end{Remark}
\subsection{Characterization of the set of admissible contracts}\label{section:characterization}
Let $(Y^\lambda, Z^\lambda)$ be the (unique) solution to BSDE \eqref{BSDE:agent:sol}. We fix $\lambda$ and the Pareto optimum $a^\star$ in $\mathcal A^\star(X,Z^\lambda,\lambda)$ (reduced to a singleton for the sake of simplicity as explained in Remark \ref{remak:choix}). We consider the following BSDE
\begin{align}\label{BSDE:agent:i}
\nonumber Y_t^i&=U_A^i(\xi^i)+ \Gamma_i(X_T)+\int_t^T \left(b(s,X_s, a^\star(s,X_s,Z_s^\lambda,\lambda))\cdot Z_s^i- k^i(s,X_s, (a^{\star}(s,X_s,Z_s^\lambda,\lambda))^{:,i})\right) ds\\
&-\int_t^T Z_s^i\cdot dX_s
\end{align}
We thus have the fundamental following Proposition and we refer to Appendix \ref{appendix:proof} for its proof
\begin{Proposition}\label{prop:caract} 
For any $i=1,\dots, N$, BSDE \eqref{BSDE:agent:i} admits a unique solution $(Y^i,Z^i) \in  \mathcal E(\mathbb R)\times \bigcap_{p\geq 1}\mathcal H^p(\mathbb R^N) $ such that
$$Y^i_0= U_0^i(\xi, a^\star(s,X_s,Z_s^\lambda,\lambda) ).$$
Besides, let $(Y^\lambda,Z^\lambda)$ be the (unique) solution of BSDE \eqref{BSDE:agent:sol}, we have  
$$Y^\lambda=\sum_{i=1}^N \lambda_i Y^i, \; Z^\lambda=\sum_{i=1}^N \lambda_i Z^i $$ and the following decomposition for any admissible contracts holds for any $1\leq i \leq N$
\begin{align}\label{caract:xi}
\nonumber U_A^i(\xi^i)&=Y_0^i-\Gamma_i(X_T)\\
&-\int_0^T \left(b(s,X_s, a^\star(s,X_s,Z_s^\lambda,\lambda))\cdot Z_s^i- k^i(s,X_s, (a^{\star}(s,X_s,Z^\lambda_s, \lambda))^{:,i})\right) ds+\int_0^T Z_s^i\cdot dX_s.
\end{align}
\end{Proposition}
\noindent In view of Proposition \ref{prop:caract} above, we have
$$Y_t^i=Y_0^i+\int_0^t  k^i(s,X_s, (a^{\star}(s,X_s,Z^\lambda_s, \lambda))^{:,i})ds+\int_0^t Z_s^i\cdot \Sigma(s,X_s)dW^\star_s, \; Y_T^i=U_A^i(\xi^i)+\Gamma_i(X_T). $$
This suggests to introduce a set of contracts $\Xi$ as the set of random variables $(U_A^i)^{(-1)}(Y_T^{i,Z^i,Y_0}-\Gamma_i(X_T))$ with $i\in\{1,\dots,N\}$ defined by
$$Y_t^{i,Z^i,Y_0}=Y_0^i+\int_0^t  k^i(s,X_s, (a^{\star}(s,X_s,Z^\lambda_s, \lambda))^{:,i})ds+\int_0^t Z_s^i\cdot \Sigma(s,X_s)dW^\star_s,\; Z^\lambda:=\sum_{i=1}^N \lambda_i Z^i$$ for any controls $((Z^i)_{1\leq i\leq N},Y_0)\in \mathbb D(\mathbb R^N)\times \left([R_0^i,+\infty)\right)_{1\leq i\leq N}$, where $\mathbb D(\mathbb R^N)$ denotes the set of controls $\mathbb F$-predictable processes $(Z^i)_{1\leq i\leq N} \in \mathcal H^2(\mathbb R^N)$ such that $(U_A^i)^{(-1)}(Y_T^{i,Z^i,Y_0}-\Gamma_i(X_T))\in \mathcal C$.

\noindent We set an $\mathbb F$-predictable process $Z$ with values in $\mathcal M_N(\mathbb R)$ such that its coefficient $Z^{i,j}$ is the $i$th element of the vector process $Z^j$. In this case we define the $\mathbb R^N$-valued process 
$$Y_t^{Z,Y_0}:=Y_0+\int_0^t  K(s,X_s, (a^{\star}(s,X_s,Z_s, \lambda)))ds+\int_0^t Z_s^{\top}\Sigma(s,X_s)dW^\star_s,$$
by setting (with an abuse of notation but justified by Proposition \ref{prop:caract}) \begin{equation}\label{def:astar}a^{\star}(s,X_s,Z_s, \lambda):= a^{\star}(s,X_s,Z^\lambda_s, \lambda)\end{equation}
with $Z^\lambda=\left(\sum_{i=1}^N \lambda_i Z^{j,i}\right)_{1\leq j\leq N}$ and $$K(s,X_s, (a^{\star}(s,X_s,Z_s, \lambda))):=(k^i(s,X_s, (a^{\star}(s,X_s,Z_s, \lambda))^{:,i}))^\top_{1\leq i\leq N}.$$
In this case, from Proposition \ref{prop:caract}, we get\footnote{The inclusion $\Xi\subset \mathcal C$ is in the definition of $\Xi$. The inclusion $\mathcal C\subset\Xi$ comes from \eqref{caract:xi} by using martingale representation.}
\begin{equation}\label{characterization}
\mathcal C=\Xi.
\end{equation}
In the following we will denote $\mathbb M_N(\mathbb R^N)$ the set of $\mathbb F$-predictable process $Z$ with values in $\mathcal M_N(\mathbb R)$ such that $Z^{:,i}\in \mathbb D(\mathbb R^N)$ for any $1\leq i\leq N$.
\subsection{The general problem of the Principal with an exogenous Planner}
Using the notation \eqref{def:astar}, we recall that the Principal solves
\begin{equation}\label{pb:primal:principal}
U_0^P=\sup_{\xi=(\xi^i)_{1\leq i\leq N} \in \mathcal C} \mathbb E^{\mathbb P^{a^\star(X, Z^\lambda,\lambda)}} \left[\mathcal U_P\left( \ell(X_T) - \sum_{i=1}^N \xi^i\right) \right],
\end{equation}
From Proposition \ref{prop:caract} and Characterization \eqref{characterization}, Problem \eqref{pb:primal:principal} becomes

\begin{equation}\label{pb:principal}
 U_0^P=\underset{Y^i_0\geq R_0^i} {\sup_{Z\in \mathbb M_N(\mathbb R^N)}}\mathbb E^{\mathbb P^{a^\star(X,Z,\lambda)}} \left[\mathcal U_P\left( \ell(X_T) - \sum_{i=1}^N (U_A^i)^{-1}(Y^{i,Z^{:,i},Y^i_0}_T-\Gamma_i(X_T))\right) \right].
\end{equation}
Since $\mathcal U_P$ and $(U_A^i)^{(-1)}$ are increasing, one can explicitly get the optimal $Y_0^i$ for any $i$ by saturating the reservation utility constrains of the Agents, \textit{i.e.}, $Y^{i,\star}_0=R_0^i$. Thus, Problem \eqref{pb:principal} becomes
\begin{align}\label{pb:principalsansY0}
 U_0^P&=\sup_{Z\in \mathbb M_N(\mathbb R^N)}\mathbb E^{\mathbb P^{a^\star(X, Z,\lambda)}} \left[\mathcal U_P\left( \ell(X_T) - \sum_{i=1}^N (U_A^i)^{(-1)}(Y^{i,Z^{:,i},R^i_0}_T- \Gamma_i(X_T))\right) \right].
\end{align}
\noindent As emphasized above and as an extension of \cite{sannikov2008continuous, cvitanic2014moral,cvitanic2017dynamic}, solving \eqref{pb:principalsansY0} remains to solve a stochastic control problem with 
\begin{itemize}
\item control variable: $Z\in \mathbb M_N(\mathbb R)$,
\item two state variables: the value of the firm $X$ and the value function of the $i$th Agent $Y^{i,Z^{:,i},R_0^i}$ for any $i\in \{1,\dots,N\}$.
\end{itemize}

\noindent We define $H^{Par}:[0,T]\times \mathbb R^N\times \mathbb R^N\times  \mathbb R^N\times \mathcal M_N(\mathbb R)\times \mathcal M_N(\mathbb R) \times \mathcal M_N(\mathbb R)\times (0,1)^N\longrightarrow \mathbb R$ by
\begin{align*}
H^{Par}(t,x,p,q,P,Q,R,\lambda)&:=\sup_{z\in \mathcal M_N(\mathbb R)}\Big\{ b(t,x, a^\star(t,x,z,\lambda))\cdot p+  K(t,x, (a^{\star}(t,x,z, \lambda)))\cdot q \\
&+\frac12 \text{Tr} \left( \Sigma(t,x) \Sigma(t,x)^\top P\right)+\frac12 \text{Tr} \left(z^\top \Sigma(t,x) \Sigma(t,x)^\top z Q\right)\\
&+\text{Tr}\left( \Sigma(t,x) \Sigma(t,x)^\top z R\right) \Big\}.
\end{align*}
We thus consider the following HJB equation associated with \eqref{pb:principalsansY0}
$$\textbf{(HJB-Par)}\begin{cases}
\displaystyle -\partial_t v(t,x,y)-H^{Par}(t,x,\nabla_x v,\nabla_y v,\Delta_{xx}v,\Delta_{xy}v,\Delta_{xy}v,\lambda)&=0,\; (t,x,y)\in [0,T)\times \R^N\times \R^N\\
v(T,x,y)=\mathcal U_P\left(\ell(x)-\sum_{i=1}^N (U_A^{i})^{(-1)}(y^i-\Gamma_i(x))\right).
\end{cases}$$
As usual, using a classical verification theorem (see for instance \cite{touzi2012optimal}) we get the following result\vspace{0.5em}

\begin{Theorem} Assume that PDE \textbf{(HJB-Par)} admits a unique solution $v$ continuously differentiable with respect to $t$ and twice continuously differentiable with respect to its spaces variables and that for any $t,x,p,q,P,Q,R,\lambda\in [0,T]\times \mathbb R^N\times \mathbb R^N\times  \mathbb R^N\times \mathcal M_N(\mathbb R)\times \mathcal M_N(\mathbb R) \times \mathcal M_N(\mathbb R)\times (0,1)^N$, the supremum in $H^{Par}(t,x,p,q,P,Q,R,\lambda)$ is attained for (at least) one $z^\star(t,x,p,q,P,Q,R,\lambda)$. Moreover, for any $Y^i_0\geq R_0^i, \; 1\leq i\leq N$, we assume that the following coupled SDE
$$\begin{cases}
\displaystyle X^\star_t&=x+\int_0^t \Sigma(s,X_s^\star)dW_s,\\
\displaystyle Y_t^{\star}&=R_0+\int_0^t  \left(K(s,X^\star_s, (a^{\star}(s,X^\star_s,z^{\star}_s, \lambda)))- (z^\star_s)^{\top} b(s,X^\star_s, (a^{\star}(s,X^\star_s,z^{\star}_s))\right)ds\\
\displaystyle&+\int_0^t (z^\star_s)^{\top}\Sigma(s,X^\star_s)dW_s,
\end{cases}$$
admits a unique solution $(X^\star,Y^\star)$ where 
$$z^\star_t:=z^\star(t,X^\star_t,\nabla_x v,\nabla_y v,\Delta_{xx}v,\Delta_{yy}v,\Delta_{xy}v,\lambda).$$
If moreover $(z^\star)^{:,i}\in \mathbb D(\mathbb R^N)$ then $$\xi:=\left((U_A^i)^{(-1)}((Y_T^{\star})^i-\Gamma_i(X^\star_T))\right)_{1\leq i\leq N}\in \mathcal C$$ is an optimal contract which solves the Principal problem \eqref{pb:primal:principal}, with
$$U_0^P=v(0,x,R_0).$$
 \end{Theorem}

\section{Application to a bidimensional linear-quadratic model}\label{section:application}

We compare the case where a Planner intervenes in the Stackelberg equilibrium between the Principal and the Agent with the case in which no-Planner acts in the linear-quadratic model developed in \cite{elie2016contracting}.

\subsection{The model and characterization of Pareto optima}
Let $a\in\mathcal M_2(A)$, $b: \mathcal M_2(\mathbb R) \longrightarrow \mathbb R^2$ and $k: \mathcal M_2 \longmapsto \mathbb R^2$.
$$dX_t= b(a_t) dt +dW_t^a$$
$a^{i,j}$ action of Agent $j$ on the project $i$ with
$$b(a)= \begin{pmatrix} a^{11}-a^{12}\\ a^{22}-a^{21}\end{pmatrix},\; k(a)=  \begin{pmatrix} \frac{k^{11}}2 |a^{11}|^2 + \frac{k^{21}}2 |a^{21}|^2\\  \frac{k^{22}}2 |a^{22}|^2 + \frac{k^{12}}2 |a^{12}|^2 \end{pmatrix}. $$
We also assume that their exists $\gamma_1,\gamma_2>0$ such that for any $i\in\{1,2\}$
$$\Gamma_i(x)=\Gamma_i\cdot x,\; x\in \mathbb R^2, \, j\neq i, $$
with $$\Gamma_1=\gamma(1, -1)^\top,\;  \Gamma_2=\gamma(-1, 1)^\top,\, \gamma\geq 0.$$
\noindent Then, for any parameter $\lambda=(\lambda_1,\lambda_2)\in (0,1)^2$ with $\lambda_1+\lambda_2=1$, we have
for any $z:=(z^{i,j})_{1\leq i,j\leq 2}\in \mathcal M_2(\mathbb R)$
$$a^{\star}(z,\lambda)=\begin{pmatrix}
\frac{\lambda_1 z^{11}+\lambda_2 z^{12}}{\lambda_1 k^{11}} &- \frac{\lambda_1 z^{11}+\lambda_2 z^{12}}{\lambda_2 k^{12}} \\
- \frac{\lambda_1 z^{21}+\lambda_2 z^{22}}{\lambda_1 k^{21}}  & \frac{\lambda_1 z^{21}+\lambda_2 z^{22}}{\lambda_2 k^{22}} 
\end{pmatrix}$$

\subsection{Problem of the Principal under the action of a general Planner} 
$$ U_0^P(\lambda):={\sup_{Z\in\mathbb M_2(\mathbb R)}} \, \mathbb E^{\mathbb P^{a^{\star}(Z,\lambda)}} \left[-e^{-R_P \left(X_T\cdot \mathbf 1_N- \xi^1-\xi^2 \right)}\right].$$
We get 
\begin{align*}
U_0^P(\lambda)&= {\sup_{Z\in \mathbb M_2(\mathbb R)}} \, \mathbb E^\star \left[-e^{-R_P \left(\int_0^T (b^1+b^2)(a^{\star}(Z,\lambda))ds +\int_0^T \mathbf 1_2\cdot  dW_s^\star - \int_0^T (k^1+k^2)(a^{\star}(Z,\lambda)) ds -\int_0^T(Z^{:,1}+Z^{:,2})\cdot dW_s^\star  \right)}\right]\\
&= {\sup_{Z\in \mathbb M_2(\mathbb R)}} \, \mathbb E^\star \left[-\mathcal E\left(-R_P\int_0^T(\mathbf 1_2 - (Z^{:,1}_s+Z^{:,2}_s)) \cdot dW_s^\star \right)e^{-R_P\int_0^T g(\lambda,Z_t)dt} \right]
\end{align*}
 where $\mathbb E^\star:=\mathbb E^{\mathbb P^{a^{\star}(Z,\lambda)}}$ and  $W^\star:= W^{a^{\star}(Z,\lambda)}$ and with, by setting $\overline \lambda :=1-\lambda$
\begin{align*}
g(\lambda, z)&=-\frac{R_P}2 \left( z^{11}+z^{12}-1\right)^2- \frac{R_P}2 \left( z^{22}+z^{21}-1\right)^2\\[0.8em]
&+\left( \frac{\lambda z^{11}+\overline \lambda z^{12}}{\lambda k^{11}}  +  \frac{\lambda z^{11}+\overline \lambda  z^{12}}{\overline \lambda  k^{12}} \right)+\left( \frac{\lambda z^{21}+\overline \lambda  z^{22}}{\overline \lambda  k^{22}}+  \frac{\lambda z^{21}+\overline \lambda  z^{22}}{\lambda k^{21}} \right)\\[0.8em]
& -\frac{k^{11}}2 \left| \frac{\lambda z^{11}+\overline \lambda  z^{12}}{\lambda k^{11}} \right|^2 -\frac{k^{21}}2 \left| \frac{\lambda z^{21}+\overline \lambda  z^{22}}{\lambda k^{21}} \right|^2 -\frac{k^{22}}2 \left|\frac{\lambda z^{21}+\overline \lambda  z^{22}}{\overline \lambda  k^{22}} \right|^2- \frac{k^{12}}2  \left|  \frac{\lambda z^{11}+\overline \lambda  z^{12}}{\overline \lambda  k^{12}}  \right|^2.
\end{align*} 
One can show after tedious but easy computations that $g$ is coercive in $z^1$ and $z^2$ and convexe. The first order conditions are then given by $$\textsc{(FOC)}\begin{cases}

R_P(z^{11}+z^{12}-1)+\left( \frac{\lambda z^{11}+\overline \lambda  z^{12}}{\lambda k^{11}}\right)+\frac{\lambda}{\overline \lambda  }\left( \frac{\lambda z^{11}+\overline \lambda  z^{12}}{\overline \lambda k^{12}}\right)&=\frac{1}{k^{11}} + \frac{\lambda}{\overline \lambda  k^{12}}\\
\left(\frac{\lambda}{\overline{\lambda}}-1\right)R_P(z^{11}+z^{12}-1)=0&\\
R_P(z^{22}+z^{21}-1)+\left( \frac{\lambda z^{21}+\overline \lambda  z^{22}}{\overline \lambda k^{22}}\right)+\frac{\overline \lambda }{\lambda}\left( \frac{\lambda z^{21}+\overline \lambda  z^{22}}{\lambda k^{21}}\right)&=\frac{1}{k^{22}} + \frac{\overline \lambda }{\lambda k^{21}}\\
\left(\frac{\lambda}{\overline{\lambda}}-1\right)R_P(z^{22}+z^{21}-1)&=0.

\end{cases}$$
This implies to distinguish $\lambda=\frac12$ or not.

\subsubsection{Pareto optima and Optimal contracts with an exogenous Planner}\label{application:exogeneous}
For any $\lambda\neq\frac12$ fixed by an exogenous Planner, (FOC) gives
 \begin{align*}
g(\lambda, z)&=\left( \frac{\lambda z^{11}+\overline \lambda z^{12}}{\lambda k^{11}}  +  \frac{\lambda z^{11}+\overline \lambda  z^{12}}{\overline \lambda  k^{12}} \right)+\left( \frac{\lambda z^{21}+\overline \lambda  z^{22}}{\overline \lambda  k^{22}}+  \frac{\lambda z^{21}+\overline \lambda  z^{22}}{\lambda k^{21}} \right)\\[0.8em]
& -\frac{k^{11}}2 \left| \frac{\lambda z^{11}+\overline \lambda  z^{12}}{\lambda k^{11}} \right|^2 -\frac{k^{21}}2 \left| \frac{\lambda z^{21}+\overline \lambda  z^{22}}{\lambda k^{21}} \right|^2 -\frac{k^{22}}2 \left|\frac{\lambda z^{21}+\overline \lambda  z^{22}}{\overline \lambda  k^{22}} \right|^2- \frac{k^{12}}2  \left|  \frac{\lambda z^{11}+\overline \lambda  z^{12}}{\overline \lambda  k^{12}}  \right|^2.
\end{align*}
By maximizing the function $g$, we get the following general Pareto optima conditions 

$$\begin{cases}
z^{11}_{Pe}&=\frac{|1-\lambda|^2k^{12}}{|\lambda|^2k^{11}+|1-\lambda|^2k^{12}}\\[0.8em]
z^{21}_{Pe}&= \frac{|1-\lambda|^2k^{22}}{|\lambda|^2k^{21}+|1-\lambda|^2k^{22}}
\\[0.8em]
z^{22}_{Pe}&= \frac{|\lambda|^2k^{21}}{|\lambda|^2k^{21}+|1-\lambda|^2k^{22}}\\
z^{12}_{Pe}&=\frac{|\lambda|^2k^{11}}{|\lambda|^2k^{11}+|1-\lambda|^2k^{12}}
\\[0.8em]
 $$a^{\star}(z_{Pe},\lambda)&=\begin{pmatrix}
\frac{\lambda z^{11}_{Pe}+(1-\lambda) z^{12}_{Pe}}{\lambda k^{11}} &- \frac{\lambda z^{11}_{Pe}+(1-\lambda) z^{12}_{Pe}}{(1-\lambda) k^{12}} \\
- \frac{\lambda z^{21}_{Pe}+(1-\lambda) z^{22}_{Pe}}{\lambda k^{21}}  & \frac{\lambda z^{21}_{Pe}+(1-\lambda) z^{22}_{Pe}}{(1-\lambda) k^{22}} 
\end{pmatrix}$$\\[0.8em]
U_0^P(\lambda)&=-e^{-R_PT g(\lambda, z_{Pe})}
\end{cases}$$

\noindent Thus, given a Pareto optimum of parameter $\lambda\in (0,1)\setminus \{\frac12\}$ chosen by the Planner, the optimal contracts given to Agent 1 and Agent 2 are respectively 
\begin{align*}
 \xi^{1,\star}&=R_0^1+\frac{T}2\left( \frac{\left|\lambda z^{11}_{Pe}+(1-\lambda) z^{12}_{Pe}\right|^2}{|\lambda|^2k^{11}} +  \frac{\left|\lambda z^{21}_{Pe}+(1-\lambda) z^{22}_{Pe}\right|^2}{|\lambda|^2 k^{21}}\right)\\
 &-T z^{11}_{Pe}\left(\frac{\lambda z^{11}_{Pe}+(1-\lambda) z^{12}_{Pe}}{\lambda k^{11}} +  \frac{\lambda z^{11}_{Pe}+(1-\lambda) z^{12}_{Pe}}{(1-\lambda) k^{12}}\right)\\
 &-T z^{21}_{Pe}\left(\frac{\lambda z^{21}_{Pe}+(1-\lambda) z^{22}_{Pe}}{(1-\lambda) k^{22} }+ \frac{\lambda z^{21}_{Pe}+(1-\lambda) z^{22}_{Pe}}{\lambda k^{21}} \right)\\
 &+(z^{11}_{Pe}-\gamma) X_T^1 +(z^{21}_{Pe}+\gamma)X_T^2.
 \end{align*}
 and
 \begin{align*}
 \xi^{2,\star}&=R_0^2+\frac{T}2\left( \frac{\left|\lambda z^{11}_{Pe}+(1-\lambda) z^{12}_{Pe}\right|^2}{|1-\lambda|^2k^{12}} +  \frac{\left|\lambda z^{21}_{Pe}+(1-\lambda) z^{22}_{Pe}\right|^2}{|1-\lambda|^2 k^{22}}\right)\\
 &-Tz^{12}_{Pe}\left(\frac{\lambda z^{11}_{Pe}+(1-\lambda) z^{12}_{Pe}}{\lambda k^{11}} +  \frac{\lambda z^{11}_{Pe}+(1-\lambda) z^{12}_{Pe}}{(1-\lambda) k^{12}}\right)\\
 &-Tz^{22}_{Pe}\left(\frac{\lambda z^{21}_{Pe}+(1-\lambda) z^{22}_{Pe}}{(1-\lambda) k^{22} }+ \frac{\lambda z^{21}_{Pe}+(1-\lambda) z^{22}_{Pe}}{\lambda k^{21}} \right)\\
 &+(z^{12}_{Pe}+\gamma) X_T^1 + (z^{22}_{Pe}-\gamma)X_T^2.
 \end{align*}

\noindent Notice that $g$ admits a maximum in $\lambda=\frac12$. We can deduce that the case $\lambda=\frac12$ coincides with a cooperative Planner with the Principal. 
\begin{Remark}\label{rem:weakpareto} In fact, notice that $\lambda\in(0,1)\longrightarrow g(\lambda,z_{Pe})$ has a continuous extension when $\lambda\to 0$ or $\lambda\to 1$, which corresponds to weak Pareto optima (see \cite[Theorem 3.1.1]{miettinen2012nonlinear}). In this case, we set
$$g(0):=g(0,z_{Pe})=\frac{1}{2k^{11}}+\frac{1}{2k^{21}} , \; g(1):=g(1,z_{Pe})= \frac{1}{2k^{22}}+ \frac{1}{2k^{12}}.$$
\end{Remark}
\subsubsection{Pareto optimum with a cooperative Planner}\label{pareto:cooperative:application}

It remains to the case where the Planner and the Principal cooperate by choosing $\lambda_1=\frac12=\lambda_2$. 

\noindent The first order conditions (FOC) induces the following Pareto optima conditions 

$$\begin{cases}
z^{11}_{Pc}+z^{12}_{Pc}&=1 \\
z^{22}_{Pc}+z^{21}_{Pc}&=1
\\[0.8em]
 $$a^{\lambda,\star}(z_{Pc})&=\begin{pmatrix}
\frac{1}{k^{11}} &- \frac{1}{ k^{12}} \\
- \frac{1}{k^{21}}  & \frac{1}{k^{22}} 
\end{pmatrix}$$\\
 U_0^P&= -e^{-R_P T g(0.5)},
\end{cases}$$
with $$g(0.5):=g(0.5,z^1_{Pc},z^2_{Pc})=\left( \frac{1}{ 2k^{11}}  +  \frac{1}{2 k^{12}} \right)+\left( \frac{ 1}{2k^{22}}+  \frac{1}{2k^{21}} \right).$$
\noindent 
We thus have the following optimal contracts for Agent 1 and Agent 2 respectively
\begin{align*}
 \xi^{1,\star}&= R_0^1+T\left( \frac{1}{2k^{11}}+ \frac{1}{2k^{21}}\right)-Tz^{11}_{Pc}\left(\frac{1}{k^{11}}+\frac{1}{k^{12}} \right)-Tz^{21}_{Pc}\left(\frac{1}{k^{21}}+\frac{1}{k^{22}} \right)\\
 &+(z^{11}_{Pc} -\gamma)X_T^1+(z^{21}_{Pc}+\gamma)X_T^2,  \end{align*}
and
\begin{align*}
 \xi^{2,\star}&= R_0^2+T\left( \frac{1}{2k^{12}}+ \frac{1}{2k^{22}}\right)-Tz^{12}_{Pc}\left(\frac{1}{k^{11}}+\frac{1}{k^{12}} \right)-Tz^{22}_{Pc}\left(\frac{1}{k^{21}}+\frac{1}{k^{22}} \right)\\
 &+(z^{12}_{Pc}+\gamma) X_T^1+(z^{22}_{Pc}-\gamma)X_T^2.  \end{align*}
\paragraph{Interpretations.}
\begin{itemize}\item[$(i)$] We find the first-best effort cf \cite[Theorem 3.3.]{elie2016contracting} which is both a Nash equilibrium in the first best model and a Pareto optimum with a cooperative Planner.
\item[$(ii)$] We get in the case of a cooperative Planner infinitely number of optimal contracts which differs clearly form the classical case and results in \cite{elie2016contracting}.
\item[$(iii)$]  If $k^{11}\leq k^{22}$ and $k^{21}\leq k^{12}$, \textit{i.e.} Agent $1$ is more efficient that Agent $2$ to manage his project and to help Agent $2$ to manage his project. Take $z^{11}_{Pc}=z^{22}_{Pc}=1$ each Agent strictly will receive the value of his own project. Fixed part is greater for Agent $1$ than for Agent $2$. 
\item[$(iv)$]  If $k^{12}\geq k^{11}$, \textit{i.e.} Agent $1$ is more efficient with project $1$ than Agent $2$, the part of the project $1$ given to Agent $1$ is greater than the part of this project given to Agent $2$ and depend of the relative performance of this Agent compared to the others. We get similar results with project $2$.
\item[$(v)$] Notice that for any Pareto coefficients, the optimal effort of each Agent is always a booster of each project. At the equilibrium, each Agent does not penalize the project of the others. This phenomenon indeed appears as soon as the appetence coefficients are different for Agent 1 and Agent 2 as emphazised in \cite{elie2016contracting}.
\item[$(vi)$] Finally, notice that at the optimum, the appetence parameter $\gamma$ does not impact the value function of the Principal and optimal efforts. More exactly, in this model the Principal proposes optimal contracts to the Agents which cancelled their ambitious in their works and such that ambitions of the Agents have no impact on the value of the Principal. Besides, we recover the suppression of the appetence for competitions of the Agent by the Principal since she penalizes each $i$th Agent with the amount $-\gamma(X_T^i- X_T^j)$, $i\neq j\in \{1,2\}$, as in \cite{elie2016contracting}.
\end{itemize}

\subsubsection{Comparison with the no-Planner model}
We now turn to a comparison between a model in which a Planner manages the effort of the Agents and a model in which Agents are rational entities who aim at finding a Nash equilibrium. Using \cite{elie2016contracting} and results of Section \ref{section:noplanner}, the Nash equilibrium is given for any $z\in \mathcal M_2(\mathbb R)$ by
$$ e^\star(z)= \begin{pmatrix}\frac{z^{11}}{k^11} &- \frac{z^{12}}{k^{12}}\\
-\frac{z^{21}}{k^{21}} & \frac{z^{22}}{k^{22}} \end{pmatrix}$$
The problem of the Principal is then to solve
$$ U_0^{P,NA}:={\sup_{\xi^1,\xi^2\in \mathcal C}} \, \mathbb E^\star \left[-e^{-R_P \left(X_T\cdot \mathbf 1_N- \xi^1-\xi^2 \right)}\right].$$
We get
\begin{align*}
U_0^P&= {\sup_{Z}} \, \mathbb E^\star \left[-\mathcal E\left(-R_P\int_0^T( \mathbf 1_2- (Z^1_s+Z^2_s)) \cdot dW_s^\star \right)e^{-R_P\int_0^T g_{NA}(Z_t,R_P)dt} \right]
\end{align*}
with
\begin{align*}
g_{NA}(z,R_P)&=-\frac{R_P}2 \left( z^{11}+z^{12}-1\right)^2- \frac{R_P}2 \left( z^{22}+z^{21}-1\right)^2\\
&+\left( \frac{ z^{11}}{ k^{11}}  +  \frac{z^{12}}{ k^{12}} +\frac{z^{21}}{k^{21}}+  \frac{z^{22}}{k^{22}} \right) -\frac{k^{11}}2 \left| \frac{ z^{11}}{k^{11}} \right|^2 -\frac{k^{21}}2 \left| \frac{z^{21}}{ k^{21}} \right|^2 -\frac{k^{22}}2 \left|\frac{z^{22}}{k^{22}} \right|^2- \frac{k^{12}}2  \left|  \frac{z^{12}}{ k^{12}}  \right|^2.
\end{align*}
Noticing that $g^{NA}$ is coercive and convexe, one get from the first order conditions the following solutions
$$\begin{cases}
z^{11}_{NA}=\frac{1+ R_Pk^{12}}{1+ R_P(k^{12}+k^{11})}\\[0.3em]
z^{12}_{NA}=\frac{1+R_Pk^{11}}{1+ R_P(k^{12}+k^{11})}\\[0.3em]
z^{22}_{NA}=\frac{1+ R_Pk^{21} }{1+ R_P(k^{21}+k^{22})}\\[0.3em]
z^{21}_{NA}=\frac{1+R_Pk^{22}}{1+ R_P(k^{21}+k^{22})}\\[0.3em]
$$ e^\star(z)= \begin{pmatrix}\frac{z^{11}_{NA}}{k^{11}} &- \frac{z^{12}_{NA}}{k^{12}}\\
-\frac{z^{21}_{NA}}{k^21} & \frac{z^{22}_{NA}}{k^{22}} \end{pmatrix}$$\\[0.3em]
U_0^{P,NA}= -e^{-R_P T g_{NA}(z,R_P)}
\end{cases}$$

In this case, optimal contracts are given by
\begin{align*}
\xi^{1,\star}&= R_0^1+T\left( \frac{1}{2k^{11}}+ \frac{1}{2k^{21}}\right)-Tz^{11}_{NA}\left(\frac{z^{11}_{NA}}{k^{11}}+\frac{z^{12}_{NA}}{k^{12}} \right)-Tz^{21}_{NA}\left(\frac{z^{21}_{NA}}{k^{21}}+\frac{z^{22}_{NA}}{k^{22}} \right)\\
&+(z^{11}_{NA}-\gamma) X_T^1+(z^{21}_{NA}+\gamma)X_T^2
\end{align*}
and
\begin{align*}\xi^{2,\star}&= R_0^2+T\left( \frac{1}{2k^{12}}+ \frac{1}{2k^{22}}\right)-Tz^{12}_{NA}\left(\frac{z^{11}_{NA}}{k^{11}}+\frac{z^{12}_{NA}}{k^{12}} \right)-Tz^{22}_{NA}\left(\frac{z^{21}_{NA}}{k^{21}}+\frac{z^{22}_{NA}}{k^{22}} \right)\\
&+(z^{12}_{NA}+\gamma) X_T^1+(z^{22}_{NA}-\gamma)X_T^2.
\end{align*}
\noindent The result below gives necessary and sufficient conditions for which the Nash equilibrium coincides with a Pareto equilibrium.

\begin{Proposition}[Suffisant and necessary conditions] \label{prop:nashpareto}Assume that $A$ is convex. 
The Nash equilibrium $e^\star$ is a Pareto optimum if and only if there exists $\lambda^\star\in (0,1)$ such that
$$\begin{cases}
\displaystyle\frac{1+ R_Pk^{12}}{1+ R_P(k^{12}+k^{11})}=\frac{\lambda^\star z^{11}_{Pe}+(1-\lambda^\star) z^{12}_{Pe}}{\lambda^\star}\\ \vspace{0.5em}
\displaystyle\frac{1+R_Pk^{22} }{1+ R_P(k^{21}+k^{22})}= \frac{\lambda^\star z^{21}_{Pe}+(1-\lambda^\star) z^{22}_{Pe}}{\lambda^\star}\\ \vspace{0.5em}
\displaystyle\frac{1+R_P k^{12}}{1+R_P k^{11}}=\frac{1+R_P k^{22}}{1+R_P k^{21}}=\frac{1-\lambda^\star}{\lambda^\star}.
\end{cases}$$
In this case $e^\star$ coincides with Pareto optimum with parameter $\lambda^\star$.\end{Proposition}

\begin{proof}
Since $A$ is convex, we notice that $a\longmapsto U_0^i(\xi,a)$ is concave since $b$ is linear and $k$ is convex with respect to $a$ and by using comparison of BSDEs (see \cite{briand2008quadratic} for instance). Thus, finding a Pareto equilibrium is equivalent to find $\lambda^\star$ such that their exists a solution of the multi-objective problem from \cite[Theorem 3.1.7]{miettinen2012nonlinear}.\vspace{0.3em}

\noindent Now, a Nash equilibrium $e^\star$ is Pareto optimal if and only if there exists $\lambda^\star\in (0,1)$ such that
$$\begin{cases}
\displaystyle\frac{1+ R_Pk^{12}}{1+ R_P(k^{12}+k^{11})}=\frac{\lambda^\star z^{11}_{Pe}+(1-\lambda^\star) z^{12}_{Pe}}{\lambda^\star}\\ \vspace{0.5em}
\displaystyle\frac{1+ R_Pk^{11}}{1+ R_P(k^{12}+k^{11})}=\frac{\lambda^\star z^{11}_{Pe}+(1-\lambda^\star) z^{12}_{Pe}}{(1-\lambda^\star)}\\
\vspace{0.5em}
\displaystyle\frac{1+R_Pk^{22} }{1+ R_P(k^{21}+k^{22})}= \frac{\lambda^\star z^{21}_{Pe}+(1-\lambda^\star) z^{22}_{Pe}}{\lambda^\star}\\ \vspace{0.5em}
\displaystyle\frac{1+R_Pk^{21} }{1+ R_P(k^{21}+k^{22})}= \frac{\lambda^\star z^{21}_{Pe}+(1-\lambda^\star) z^{22}_{Pe}}{1-\lambda^\star}.\end{cases}$$
By rewriting these conditions, we get the result.
\end{proof}
\paragraph{Toy models in which Nash equilibrium are Pareto optimal.} \begin{itemize}
\item The Nash equilibrium is Pareto efficient with $\lambda^\star=\frac12$ if and only if $R_P=0$ (the Principal is risk-neutral).
\item Assume that $k^{11}=k^{21}$ and $k^{22}=k^{12}$, \textit{i.e.} the cost of the effort of the Agent $i$ is the same towards the project $i$ or the project $j$. In this case, we necessarily have
$$\lambda^\star=  \frac{1+R_Pk^{11}}{2+R_P(k^{11}+k^{22})}\in (0,1).$$ Then, the Nash equilibrium $e^\star(\xi)$ coincides with a Pareto optimum with parameter $\lambda^\star$ if
and only if $$\frac{1+ R_Pk^{22}}{1+ R_P(k^{22}+k^{11})}=\frac{ (1-\lambda^\star)^2k^{22}+(1-\lambda^\star)\lambda^\star k^{11}}{|\lambda^\star|^2k^{11}+ (1-\lambda^\star)^2 k^{22}}. $$
\end{itemize}

\noindent We now turn to a last interesting results remaining to say that the Principal can benefit from the action of a Planner to broke the Stackelberg game between her and her Agents. \vspace{0.5em}
\begin{Proposition}[Advantage of adding a Planner]\label{advantage:prop}
There exists a (non-empty) set of parameters $\Lambda\subset (0,1)$ such that $U_0^{P , NA}\leq U_0^P\left(\lambda\right),$ for any $\lambda\in \Lambda$, \textit{i.e.} adding a Planner can improve the value of the Principal.\vspace{0.3em}

\noindent Moreover,  $U_0^{P , NA}>U_0^P\left(0\right)$ and $U_0^{P , NA}>U_0^P\left(1\right)$. In other words, a Planner choosing weak Pareto optima does not improve the value function of the Principal compared to a no-Planner model.

\end{Proposition}

\begin{proof}
For any $z\in \mathcal M_2(\mathbb R)$, notice that $g_{NA}(z_{NA},R_P)\leq g(0.5)$. Thus,
$$U_0^{P NA}= -e^{-R_P T g_{NA}(z,R_P)}\leq -e^{-R_Pg(0.5)}=U_0^P\left(\frac12\right).$$ Thus, by using the continuity with respect to $\lambda$ of $\lambda\longmapsto -e^{-R_Pg(\lambda)}$, we deduce that there exists a set $\Lambda \subset (0,1)$ such that $\frac12 \in \Lambda$ and $U_0^{P NA}\leq U_0^P(\lambda)$ for any $\lambda\in \Lambda$.\vspace{0.5em}

\noindent A tedious computation shows that $R>0\longmapsto g_{NA}(z_{NA},R)$ is decreasing, and 
\begin{align*}
\lim_{R\to+\infty}g_{NA}(z_{NA},R)&=\frac{\frac{k^{12}}{k^{11}+k^{12}}}{k^{11}}\left(1-\frac{k^{12}}{2(k^{11}+k^{12})} \right) + \frac{\frac{k^{11}}{k^{11}+k^{12}}}{k^{12}}\left(1-\frac{k^{11}}{2(k^{11}+k^{12})} \right)\\
&+\frac{\frac{k^{22}}{k^{22}+k^{21}}}{k^{21}}\left(1-\frac{k^{22}}{2(k^{22}+k^{21})} \right) + \frac{\frac{k^{21}}{k^{22}+k^{21}}}{k^{22}}\left(1-\frac{k^{21}}{2(k^{22}+k^{21})} \right)\end{align*} Thus, by using Remark \ref{rem:weakpareto}, easy computations directly give
$\lim_{R\to+\infty}g_{NA}(z_{NA},R)> \frac{1}{2k^{11}} +\frac{1}{2k^{21}}=g(0) $ and $\lim_{R\to+\infty}g_{NA}(z_{NA},R)>\frac{1}{2k^{12}} +\frac{1}{2k^{22}}=g(1).$ 
\end{proof}
\noindent We illustrate this proposition with the Figure 1 below. Notice that $\lambda\longmapsto U_0^P(\lambda)$ has the same monotonicity than $\lambda \longmapsto g(\lambda,z_{Pe\text{ or }Pc})$. Thus comparing $g(\lambda,z_{Pe\text{ or }Pc})$ with $g_{NA}(R,z_{NA})$ is enough to compare the value functions of the Principal with and without a Planner. We notice in particular that the length of the set $\Lambda$ is decreasing with the risk aversion of the Principal. 
\begin{figure}[H]\caption{Comparison between $l\in (0,1)\longmapsto g(l,z_{Pe\text{ or }Pc})$ (see the blue curve and the blue point for $l=\frac12$) and $g_{NA}(R,z_{NA})$ for $R\in \{0.1,0.25,0.5,1,50\}$ (see the horizontal red curves) by adding the weak Pareto optima for $l\in \{0,1\}$ (see the green points). Cost Parameters chosen are $k^{11}=2, \, k^{22}=5, \, k^{12}=1,\, k^{21}=10.$ In this case for any $R>0$, we have $g_{NA}(z_{NA}, R)\geq \lim\limits_{R\to+\infty}g_{NA}(z_{NA}, R)=0.7>\text{ max }\left(g(0),g(1)\right)=g(1)=0.6$. Thus, the set of Pareto optima $\Lambda$ ensuring a better value for the Principal compared to the case without adding a Planner is composed by all the $\lambda\in (0,1)$ such that the blue curve is above the red curve with corresponding risk-aversion parameter $R>0$.}
\centering \includegraphics[scale=0.6]{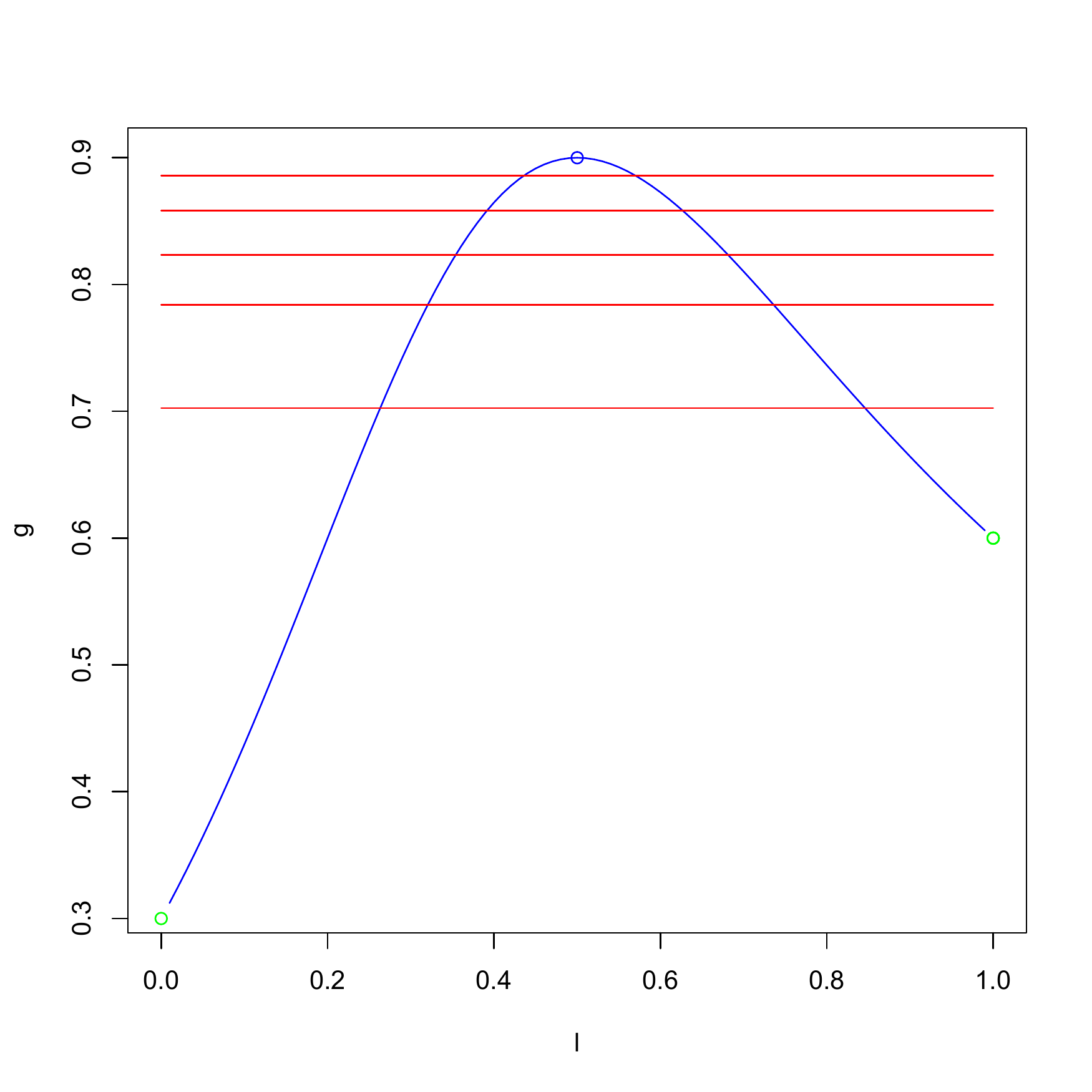}
\end{figure}
\section{Conclusion}
In this paper, we have proved that in a linear-quadratic model there exists configurations such that Nash equilibria are Pareto efficient (see Proposition \ref{prop:nashpareto}). We also have seen that it is sometimes advantageous from the Principal point of view, to add a Planner in the model who imposes Pareto optimal actions to the Agents, compared to the case in which the Agents are rational, non-cooperative, and compute themselves their best-reaction effort (see Proposition \ref{advantage:prop}). On the contrary, if the Planner chooses a weak Pareto optimum, then the value of the Principal is always worse than a no-Planner model. We would like to propose some extensions to this work for the path of future researches which might also be worth investigating.\vspace{0.3em}

\noindent First extension will consist in the studying of classical exponential utilities for the Agents. The main issue of this work is indeed that we assume that Agents are risk neutral in the application in order to make all the computations. In this case, we do not fit with separable utilities and the scalarization method proposed to solve the problem does not work. An other approach which could be considered should be to use other methods to find Pareto optima, see for instance \cite{miettinen2012nonlinear}, more tractable for exponential utilities.\vspace{0.3em}

\noindent The second extension is directly linked to Remark \ref{remak:choix} since the natural question which is not considered in the present work is  to wonder what is happening if the Planner is not exogenous in the sense that he is also hired by the Principal. In this case, the Planner has to be seen as a \textit{consulting mediator} which aims at finding Pareto optima and which is remunerated by the Principal. This required to modify the classical Stackelberg-type approach by adding this third player in the game. 

\section{Acknowlegments}
The author is very grateful to Fran\c{c}ois Delarue for having suggested to do this study during a discussion and for his advices together with Dylan Possama\"i for his general advices in the writing of this paper. The author also thanks Miquel Oliu-Barton and Patrick Bei{\ss}ner for discussions on welfare economics and Pareto efficiency.


\newpage
\appendix
\section{The mathematical model and general notations}\label{appendix:model}
In this section we set all the notations and spaces used in this paper in the paragraphs "General notations" and "Spaces" respectively. We also defined mathematically the considered model in Paragraph "General model and definitions" and we put the general assumptions which are supposed to be satisfied in all the paper.
 
\noindent {\bf General notations.} Let $\R$ be the set of reals. Let $m$ and $n$ be two positive integers. We denote by $\mathcal M_{m,n}(\mathbb R)$ the set of matrices with $m$ rows and $n$ columns, and simplify the notations when $m=n$, by using $\mathcal M_{n}(\mathbb R):=\mathcal M_{n,n}(\mathbb R)$. We denote by ${\rm {Id}}_n\in \mathcal M_n(\R)$ the identity matrix of order $n$ and $0_N\in \mathcal M_{n}(\mathbb R)$ the zero-valued coefficients matrix. For any $M\in \mathcal M_{m,n}(\mathbb R)$, we define $M^\top \in \mathcal M_{n,m}(\R)$ as the usual transpose of the matrix $M$. We will always identify $\mathbb R^n$ with $\mathcal M_{n,1}(\mathbb R)$. For any matrix $M\in \mathcal M_n(\mathbb R)$ and for any $1\leq i\leq n$ we denote by $M^{:,i}$ its $i$th column and similarly we denote by $M^{i,:}$ its $i$th row. Besides, for any $x\in\mathbb R^n$, we denote its coordinates by $x^1,\dots,x^n$. We denote by $\|\cdot\|$ the Euclidian norm on $\mathbb R^n$, when there is no ambiguity on the dimension. The associated inner product between $x\in\mathbb R^n$ and $y\in\mathbb R^n$ is denoted by $x\cdot y$. We also denote by $\mathbf 1_n$ the $n-$dimensional vector $(1,\dots,1)^\top$. Similarly, for any $x\in\mathbb R^n$, we define for any $i=1,\dots,n,$ $x^{-i}\in\mathbb R^{n-1}$ as the vector $x$ without its $i$th component, that is to say $x^{-i}:=(x^1,\dots,x^{i-1},x^{i+1},\dots,x^n)^\top$. For any $(a,\tilde a)\in \mathbb R\times\mathbb R^{n-1}$, and any $i=1,\dots,n$, we define the following $n-$dimensional vector
$$a\otimes_i \tilde a:= (\tilde a^1,\dots, \tilde a^{i-1}, a, \tilde a^i ,\dots, \tilde a^{n-1}).$$
For any $(\alpha,\tilde M)\in \mathbb R^n\times\mathcal M_{n,n-1}(\mathbb R)$, and any $i=1,\dots,n$, we define the following $n-$dimensional matrix
$$\alpha\otimes_i \tilde M:= (\tilde M^{:,1},\dots, \tilde M^{:,i-1}, \alpha, \tilde M^{:,i} ,\dots, \tilde M^{:,n-1}).$$
\noindent For any Banach space $(E,\|\cdot \|_E)$, let $f$ be a map from $E\times \mathbb R^n$ into $\mathbb R$. For any $x\in E$, we denote by $\nabla_a f(x,a)$ the gradient of $a\longmapsto f(x,a),$ and we denote by $\Delta_{a a} f(x,a)$ the Hessian matrix of $a\longmapsto f(x,a)$ and similarly we denote by $\nabla_x f(x,a)$ the gradient of $x\longmapsto f(x,a),$ and we denote by $\Delta_{xx} f(x,a)$ the Hessian matrix of $x\longmapsto f(x,a)$. Finally, we denote by $U_A^{(-1)}$ the inverse function of any $U_A:E\longrightarrow \mathbb R$ if it exists.\vspace{1em}

\noindent {\bf Spaces.}
\noindent For any finite dimensional normed space $(E,\No{\cdot}_E)$, $ \mathcal P(E)$ (resp. $\mathcal P_r(E)$) will denote the set of $E-$valued, $\mathbb F-$adapted processes (resp. $\mathbb F-$predictable processes) and for any $p\geq 1$ we set
 \begin{align*}
\mathcal E(E)&:= \left\{Y\in \mathcal P(E),\ \text{c\`adl\`ag, such that for all }p\geq 1,\,  \mathbb E\left[ \exp\left(p  \sup_{t\in [0,T]}\|Y_t\|_E \right)\right]<+\infty\right\},\\
\mathcal S^\infty(E)&:= \left\{Y\in \mathcal P(E),\ \text{c\`adl\`ag, such that } \| Y\|_{\mathcal S^\infty(E)}:=  \sup_{t\in [0,T]}\|Y_t\|_E<+\infty\right\},\\
 \mathcal H^p(E)&:= \left\{Z\in \mathcal P_r(E),\; \| Z\|_{\mathcal H^p(E)}^p:= \mathbb E\left[\left(\int_0^T \|Z_t\|_E^2dt\right)^{p/2}\right]<+\infty\right\}\\
  \mathcal H^p_{\text{BMO}}(E)&:= \left\{Z\in \mathcal H^p(E),\; \exists C>0,\, \forall \tau \in \mathcal T,\; \mathbb E\left[\int_\tau^T \|Z_s\|^2ds\Big| \mathcal F_\tau \right]\leq C^2,\, \mathbb P-a.s.\right\}.
  \end{align*}
$\mathcal H^p_{\text{BMO}}(E)$ is the so-called set of predictable processes $Z$ such that the stochastic integral is a BMO-martingale. We refer to \cite{kazamaki2006continuous} for more details on this space and properties related to this theory. We also denote $\mathcal E(M)$ the classical Doleans-Dade exponential of any local $\mathbb F$-martingale $M$. \vspace{1em}

\noindent {\bf General model and definitions.}
\noindent We assume that the following technical assumptions hold in this paper\vspace{0.3em}
\begin{Assumption}\label{assumption:bk}
For any $i=1,\dots, N$ and any $(t,x)\in [0,T]\times \mathbb R^N$, the map $a\in \mathbb R^N \longmapsto b^i(t,x,a)$ is continuously differentiable, the map $(t,x)\in [0,T]\times \mathbb R^N \longmapsto b^i(t,x,a)$ is $\mathbb F-$progressive and measurable for any $a\in \mathbb R^N$ and we assume that their exists a positive constant $C$ such that for any $(t,x,a)\in [0,T]\times \mathbb R^N\times \mathbb R^N$
\begin{equation}
|b^i(t,x,a)|\leq C(1+\|x\|+\|a\|),\; \| \nabla_a b^i(t,x,a)\| \leq C.
\end{equation}

\noindent For any $i=1,\dots, N$ the map $(t,x)\longmapsto k^i(t,x,a)$ is $\mathbb F-$progressive and measurable. Moreover, the map $a\longmapsto k^i(t,x,a)$ is increasing, convex and continuously differentiable for any $(t,x)\in [0,T]\times \mathbb R^N$. Assume also that there exists $\kappa>0, l\geq 1, \underline m,m>0$ such that 
\begin{equation}\label{condition:coef}{\frac{l+m}{\underline m+1-l}} \vee {\frac{\underline m+2-l}{\underline m+1-l}}\leq 2,\end{equation}
and for any $(t,x,a)\in [0,T]\times \mathbb R^N\times\times \mathbb R^N$
$$0\leq k^i(s,x,a)\leq C \Bigg(1+\|x\|+\|a\|^{l+m}\Bigg), $$
$$ \|\nabla_ak^i(s,x,a)\|\geq \kappa  \|a\|^{\underline m}, \text{ and } \overline{\lim}_{|a|\to \infty} \frac{k^i(s,x,a)}{\|a\|}=+\infty.$$
 
\end{Assumption}
\begin{Remark}
In the classical linear-quadratic framework, $l=m=\underline m=1$ and Condition \eqref{condition:coef} holds. This condition will be fundamental to prove Theorem \ref{thm:pbPlanner} (see the Appendix \ref{appendix:proof}). 
\end{Remark}
\vspace{0.5em}

\begin{Assumptionn}[$\textbf{G}$]
For any $p\geq 1$ and for any $1\leq i\leq N$, $$ \mathbb E\left[ \exp\left( p\Gamma^i(X_T)\right)\right]<+\infty.$$
\end{Assumptionn}
\vspace{0.5em}

\noindent In order to define a probability $\mathbb P^a$ equivalent to $\mathbb P$ such that $W^a:= W-\int_0^T \Sigma(s,X_s)^{-1} b(s,X_s,a_s) ds$ is a Brownian motion under $\mathbb P^a$, we need to introduce the set $\mathcal A$ of admissible effort. An $\mathbb F$-adapted and $\mathcal M_N(A)$-valued process $\alpha$ is said admissible if $\left(\mathcal E\left(\int_0^t  b(s,X_s,\alpha_s)\cdot \Sigma(s,X_s)^{-1}dW_s \right)\right)_{t\in [0,T]} $ is an $\mathbb F$-martingale and if for the same $l,m$ appearing in Assumption \ref{assumption:bk} we have for any $1\leq i\leq N$ and any $p\geq 1$
$$\mathbb E\left[\exp\left(p\int_0^T \| \alpha^{:,i}_s\|^{l+m} ds \right) \right]<+\infty. $$
We also define for any $\mathbb R^{N-1}$-valued and $\mathbb F$-adapted process $a^{:,-i}$ the set $\mathcal A^i(a^{:,-i})$ by
\begin{equation}\label{def:calAi}\mathcal A^i(a^{:,-i}):=\left\{\alpha, \text{ $\mathbb R^N$-valued and $\mathbb F$-adapted s.t. } \alpha \otimes_i a^{:,-i}\in \mathcal A \right\}.\end{equation}
\noindent Now, in order to use the theory of quadratic BSDE and apply the result of for instance \cite{briand2006bsde,briand2008quadratic}, we must define the set of admissible contracts $\mathcal C$ as the set of $\mathcal F_T$-measurable and $\mathbb R^N$-valued random variable $\xi$ satisfying the constrain \eqref{reservation:constrain} and such that
$$\mathbb E\left[\exp\left(p\sum_{i=1}^N |U_A^i(\xi^i)| \right) \right]<+\infty. $$

\section{Multidimensional quadratic growth BSDEs: application of the results in \cite{harter2016stability}}\label{appendix:multidimBSDE}
The aim of this section is to provide conditions ensuring that Theorem \ref{thm:multiqgbsde} holds, by giving a set of admissible contracts for which there exists a solution to the multi-dimensional BSDE \eqref{edsr:max:nash}. We consider the set of admissible contract
$$\mathcal C_{NA}:=\left\{ \xi\in \mathcal C,\; \xi\in \mathbb D^{1,2},\, \|D\xi\|_{\mathcal S^\infty(\mathbb R^N)}<+\infty\right\},$$
where $\mathbb D^{1,2}$ is the classical of random variables which are Malliavin differentiable.\vspace{0.1em}

\noindent Recall that under Assumption \ref{assumption:bk}, from \cite[Lemma 4.1]{emp16}, for any $(s,x,z)\in [0,T]\times \mathbb R^N\times \mathbb R^N$ there exists 
$$a^\star_{NA}(s,x,z)\in \underset{a\in \mathcal{A}^i({a^{:,-i}})}{\text{ arg max }} \left\{b(s,x, a\otimes_i a^{:,-i})\cdot z- k^i(s,x,a)\right\},$$
satisfying
\begin{equation*}
\|a^\star_{NA}(s,x,z)\|\leq C_a\left( 1+\|z\|^{\frac{1}{\underline m+1-l}}\right).
\end{equation*}
We assume that
\begin{itemize}
\item[$(\textbf{Db})$] The map $x\longmapsto b^i(s,x,a)$ is continuously differentiable such that there exists $C_b>0$ with
$$\|\nabla_x b^i(s,x,a)\|\leq C_b, $$
\item[$(\textbf{Dk})$] The map $x\longmapsto k^i(s,x,a)$ is continuously differentiable such that there exists $C_k>0$ with
$$\|\nabla_x k^i(s,x,a)\|\leq C_k(1+\|a\|^{l+m}),$$
and moreover
$$\|\nabla_a k^i(s,x,a)\| \leq C_k \Bigg(1+|a|^{l+m-1}\Bigg), $$
\item[$(\textbf{Da})$] The map $x\longmapsto a^\star(s,x,z)$ is continuously differentiable such that there exists $C_a>0$ with
$$|(a^{\star})^{:,i}(s,x,z)|+ |(a^{\star})^{j,:}(s,x,z)|\leq C_a(1+\|z\|^\frac{1}{\underline m-1+l}).$$
\end{itemize}
Thus, for any $(t,x,z)\in [0,T]\times \mathbb R^N\times \mathcal M_N(\mathbb R)$
\begin{align*}
\nabla_x f^{\star,i}_{NA}(t,x,z)&= \sum_{j=1}^N\Big[\left( \nabla_x b^j(t,x, (a^\star)^{j,:})+ \nabla_a b^j(t,x, (a^\star)^{j,:}) \nabla_x (a^\star)^{j,:}\right) z^{ji}\\
&-\nabla_x k^i(t,x,(a^\star_{NA})^{:,i})-\nabla_a k^i(t,x,(a^\star_{NA})^{:,i})\nabla_x (a^\star)^{j,:}\Big]. 
\end{align*}
\begin{Lemma}\label{lemma:qgbsde}
Under Assumptions \ref{assumption:bk}, $\textbf{(Db),\,(Dk),\,(Da)}$ and assume that $l+m\leq 2(\underline m-1+l)$, we have for any $(s,x,z)\in [0,T]\times \mathbb R^N\times \mathcal M_N(\mathbb R)$
$$\|\nabla_x f^{\star,i}_{NA}(t,x,z)\|\leq  C_f(1+\|z\|^2),$$
with $C_f:=C_b+2CC_a+ 2C_kN(1+|C_a|^{l+m-1}(1+C_a))$.
\end{Lemma}
\begin{proof} We compute directly
\begin{align*}
\|\nabla_x f^{\star,i}_{NA}(t,x,z)\|&\leq \sum_{j=1}^N\Big[\left( \|\nabla_x b^j(t,x, (a^\star)^{j,:})\|+\| \nabla_a b^j(t,x, (a^\star)^{j,:}) \nabla_x (a^\star)^{j,:}\|\right) |z^{ji}|\\
&+\|\nabla_x k^i(t,x,(a^\star_{NA})^{:,i})\|+\|\nabla_a k^i(t,x,(a^\star_{NA})^{:,i})\nabla_x (a^\star)^{j,:}\|\Big]\\
&\leq C_b  \sum_{j=1}^N |z^{ji}|+CC_a(1+\|z\|^\frac{1}{\underline m-1+l}) \sum_{j=1}^N |z^{ji}|+  C_kN (1+|C_a|^{l+m}(1+\|z\|^\frac{l+m}{\underline m-1+l}))\\
&+C_k N \Bigg(1+|C_a|^{l+m-1}(1+\|z\|^\frac{l+m-1}{\underline m-1+l})\Bigg)\\
&\leq C_b  \|z\|+CC_a(1+\|z\|^\frac{1}{\underline m-1+l}) \|z\|+  C_kN (1+|C_a|^{l+m}(1+\|z\|^\frac{l+m}{\underline m-1+l}))\\
&+C_k N \Bigg(1+|C_a|^{l+m-1}(1+\|z\|^\frac{l+m-1}{\underline m-1+l})\Bigg)\\
&\leq C_b+2CC_a+ 2C_kN(1+|C_a|^{l+m-1}(1+C_a))\\
&+ \|z\|^2\left(C_b+2CC_a+C_kN|C_a|^{l+m-1}(1+C_a) \right)\\
&\leq C_f(1+\|z\|^2).
\end{align*}
\end{proof}
\noindent We now introduce a localisation of $f^\star$ defined by $f^{\star}_M(s,x,z):=f^\star(s,x,\rho_M(z))$ where $\rho_M$ satisfies projection properties on a ball centred on $0_{N}$ with radius $M$. We denote by $(Y^M,Z^M)$ the unique solution of BSDE \eqref{edsr:max:nash} with generator $f_M^\star$, which fits the classical Lipschitz BSDE framework. For all $p>1$ we set the following assumption\vspace{0.3em}

\noindent(\textbf{BMO,p}) There exists a positive contant $K$ such that\footnote{The positive constant $C_p'$ comes from Burkholder-Davis-Gundy inequality, see Section 1.1 paragraph \textit{Inequalities-BDG} in \cite{harter2016stability} for more details.}
\begin{itemize}
\item[$(i)$] $KC_fC_p'<\frac12,$
\item[$(ii)$] $\sup_{M\in \mathbb R^M}\left\|\int_0^T (Z^M_s)^\top\Sigma_sdW_s \right\|_{\text{BMO}}\leq K$.
\end{itemize}

\noindent As a direct\footnote{After having discussed with Jonathan Harter, assumption \textbf{(Df,b), (iii)} in \cite{harter2016stability} is not a \textit{canonical} assumption for the proof of \cite[Theorem 3.1]{harter2016stability}. Indeed, in this paper, the authors needs this assumption only to ensure that the BSDE with projector $f_M$ admits Malliavin differentiable solution as an application of \cite{el1997backward}. However, as explained in \cite[Section 5 and Section 6]{mastrolia2014malliavin} this assumption can be removed in the Markovian case. The author thanks Jonathan Harter for this clarification.} consequence of \cite[Theorem 3.1]{harter2016stability} and Lemma \ref{lemma:qgbsde} we have the following proposition\vspace{0.3em}

\begin{Proposition}\label{prop:existence:qgbsde}
Let $p>1$ and let Assumptions \ref{assumption:bk} $(\textbf{Db}), (\textbf{Dk}), (\textbf{Da}), (\textbf{BMO,p})$ be true with $l+m\leq 2(\underline m-1+l)$, then for any $\xi\in \mathcal C_{NA}$, the N-dimensional quadratic BSDE \eqref{edsr:max:nash} has a unique solution in the sense of Definition \ref{def:sol:multidimBSDE}.
\end{Proposition}
\section{Technical proofs of Section \ref{section:MOOP}}\label{appendix:proof}
First recall that under Assumption \ref{assumption:bk}, from \cite[Lemma 4.1]{emp16}, for any $(s,x,z,\lambda)\in [0,T]\times \mathbb R^N\times \mathbb R^N\times (0,1)^N$ there exists 
$$a^\star(s,x,z,\lambda)\in \underset{a\in \mathcal M_N(A)}{\text{argmax}}\left\{ b(s, x, a)\cdot z - k^\lambda(s,x, a)\right\},$$
satisfying
\begin{equation}\label{growth:astar}
\|a^\star(s,x,z,\lambda)\|\leq C_a\left( 1+\|z\|^{\frac{1}{\underline m+1-l}}\right).
\end{equation}
We now turn to the proofs of the main results in Section \ref{section:MOOP}
\begin{proof}[Proof of Lemma \ref{lemma:bsdeagent}] The fact that BSDE \eqref{bsde:agent:a} admits a unique solution $(Y^{\lambda, a}, Z^{\lambda,a})\in \mathcal E(\R)\times \bigcap_{p\geq 1} \mathcal H^p(\mathbb R^N)$ is a direct consequence of \cite{briand2008quadratic} using the definition of $\mathcal C,\, \mathcal A$ and Assumptions \ref{assumption:bk} and $\textbf{(G)}$ in Appendix \ref{appendix:model}. Now by changing the probability in BSDE \eqref{bsde:agent:a} and by taking the conditional expectation we directly get $Y_t^{\lambda,a}=u_t^\lambda(\xi,a)$.
\end{proof}
\begin{proof}[Proof of Theorem \ref{thm:pbPlanner}] Using \eqref{growth:astar} together with Condition \eqref{condition:coef}, the definition of $\mathcal C,\, \mathcal A$ and Assumptions \ref{assumption:bk} and $\textbf{(G)}$, we obtain from (for instance) \cite{briand2008quadratic} the existence and uniqueness of a solution $(Y^\lambda,Z^\lambda)\in \mathcal E(\R)\times \bigcap_{p\geq 1} \mathcal H^p(\mathbb R^N)$ of BSDE \eqref{BSDE:agent:sol}. Now, by using a comparison theorem for BSDE \eqref{bsde:agent:a}  (see for instance \cite[Theorem 5]{briand2008quadratic}), we deduce that $Y_0^\lambda=u_0^\lambda(\xi)$ and that any $a^\star$ in $\mathcal A^\star(X,Z^\lambda,\lambda)$ is Pareto optimal from Proposition \ref{prop:pareto}.
\end{proof}

\begin{proof}[Proof of Proposition \ref{prop:caract}] Let $(Y^\lambda,Z^\lambda)\in \mathcal E(\R)\times \bigcap_{p\geq 1} \mathcal H^p(\mathbb R^N)$ be the solution of BSDE \eqref{BSDE:agent:sol} from Theorem \ref{thm:pbPlanner}. The fact that BSDE \eqref{BSDE:agent:i} admits a unique solution $(Y^i,Z^i)$ in $\mathcal E(\R)\times \bigcap_{p\geq 1} \mathcal H^p(\mathbb R^N)$ is a direct consequence of the definitions of $\mathcal C$ and $\mathcal A$. Changing the probability by using Girsanov's theorem, we directly get
$$Y^i_0= U_0^i(\xi, a^\star(s,X_s,Z_s^\lambda,\lambda) ).$$
Denote by $\mathcal Y^\lambda:=\sum_{i=1}^N Y^i$ and $\mathcal Z^\lambda:=\sum_{i=1}^N Z^i$, we notice that the pair of process $(\mathcal Y^\lambda,\mathcal Z^\lambda)$ is solution of 
$$\mathcal Y^\lambda_t=U_A^\lambda(\xi)+\int_t^T \left(b(s,X_s, a^\star(s,X_s,Z_s^\lambda,\lambda))\cdot \mathcal Z_s^\lambda- k^\lambda(s,X_s, (a^{\star}(s,X_s,Z_s^\lambda,\lambda))\right) ds-\int_t^T \mathcal Z_s^\lambda\cdot dX_s.$$
Using the uniqueness of the solution $(Y^\lambda,Z^\lambda)$ of BSDE \eqref{BSDE:agent:sol}, we deduce that $\mathcal Y^\lambda=Y^\lambda$ and $\mathcal Z^\lambda=Z^\lambda. $
\end{proof}

\begin{thebibliography}{10}

\bibitem{Acemoglu}
D. Acemoglu and A. Simsek.
\newblock Moral hazard and efficiency in general equilibrium with anonymous trading
\newblock{\em MIT Department of Economics Working Paper No. 10-8}, 2010.

\bibitem{arnott1991welfare}
R.~Arnott and J.~Stiglitz.
\newblock The welfare economics of moral hazard.
\newblock In {\em Risk, information and insurance}, pages 91--121. Springer,
  1991.

\bibitem{arrow1954existence}
K.~J Arrow and G.~Debreu.
\newblock Existence of an equilibrium for a competitive economy.
\newblock {\em Econometrica: Journal of the Econometric Society}, pages
  265--290, 1954.

\bibitem{briand2006bsde}
P.~Briand and Y.~Hu.
\newblock {BSDE} with quadratic growth and unbounded terminal value.
\newblock {\em Probability Theory and Related Fields}, 136(4):604--618, 2006.

\bibitem{briand2008quadratic}
P.~Briand and Y.~Hu.
\newblock Quadratic {BSDE}s with convex generators and unbounded terminal
  conditions.
\newblock {\em Probability Theory and Related Fields}, 141(3-4):543--567, 2008.

\bibitem{cvitanic2014moral}
J.~Cvitani{\'c}, D.~Possama{\"\i}, and N.~Touzi.
\newblock Moral hazard in dynamic risk management.
\newblock {\em Management Science}, to appear, 2014.

\bibitem{cvitanic2017dynamic}
J.~Cvitani{\'c}, D.~Possama{\"\i}, and N.~Touzi.
\newblock Dynamic programming approach to principal--agent problems.
\newblock {\em arXiv preprint arXiv:1510.07111v3}, 2017.

\bibitem{cvitanic2012contract}
J.~Cvitani{\'c} and J.~Zhang.
\newblock {\em Contract theory in continuous--time models}.
\newblock Springer, 2012.

\bibitem{el1997backward}
N.~El~Karoui, S.~Peng, and M.-C. Quenez.
\newblock Backward stochastic differential equations in finance.
\newblock {\em Mathematical Finance}, 7(1):1--71, 1997.

\bibitem{emp16}
R.~\'Elie, T.~Mastrolia, and D.~Possama{\"\i}.
\newblock A tale of a principal and many many agents.
\newblock {\em arXiv preprint arXiv:1608.05226}, 2016.

\bibitem{elie2016contracting}
R.~\'Elie and D.~Possama{\"\i}.
\newblock Contracting theory with competitive interacting agents.
\newblock {\em arXiv preprint arXiv:1605.08099}, 2016.

\bibitem{figaro}
A.~Feertchak and G.~Poingt.
\newblock Pr\'esidentielle : faut-il remettre en cause les 3\% de d\'eficit
  public\nobreakdash ?
\newblock {\em Le Figaro}, March 2017.

\bibitem{libe}
B.~Fofana.
\newblock La r\`egle des 3\% de d\'eficit des \'etats est-elle un "non-sens"?
\newblock {\em Lib\'eration}, February, 2017.

\bibitem{frei2011financial}
C.~Frei and G.~Dos~Reis.
\newblock A financial market with interacting investors: does an equilibrium
  exist?
\newblock {\em Mathematics and financial economics}, 4(3):161--182, 2011.

\bibitem{harter2016stability}
J.~Harter and A.~Richou.
\newblock A stability approach for solving multidimensional quadratic {BSDE}s.
\newblock {\em arXiv preprint arXiv:1606.08627}, 2016.

\bibitem{holmstrom1982moral}
B.~Holmstr{\"o}m.
\newblock Moral hazard in teams.
\newblock {\em The Bell Journal of Economics}, 13(2):324--340, 1982.

\bibitem{holmstrom1987aggregation}
B.~Holmstr{\"o}m and P.~Milgrom.
\newblock Aggregation and linearity in the provision of intertemporal
  incentives.
\newblock {\em Econometrica}, 55(2):303--328, 1987.

\bibitem{hu2005utility}
Y.~Hu, P.~Imkeller, M.~M{\"u}ller.
\newblock Utility maximization in incomplete markets.
\newblock {\em The Annals of Applied Probability}, 15(3):1691--1712, 2005.

\bibitem{kazamaki2006continuous}
N.~Kazamaki.
\newblock {\em Continuous exponential martingales and BMO}.
\newblock Springer, 2006.

\bibitem{telegraph}
M.~Khan.
\newblock How the european central bank became the real villain of greece's
  debt drama.
\newblock {\em The Telegraph}, May 2015.

\bibitem{keun2008optimal}
H.K.~Koo, G.~Shim, and J.~Sung.
\newblock Optimal multi--agent performance measures for team contracts.
\newblock {\em Mathematical Finance}, 18(4):649--667, 2008.


\bibitem{laffont1996incentives}
J.J. Laffont, and J. Tirole.
\newblock A Theory of Incentives in Procurement and Regulation.
\newblock{\em MIT press,} 1993.

\bibitem{laffont2001incentive}
J.J. Laffont, and D. Martimort.
\newblock The theory of incentives: the principal-agent model.
\newblock{\em Princeton University Press, Princeton}, 2001.

\bibitem{laufer2008maastricht}
N.~KA L{\"a}ufer.
\newblock Die maastricht-kriterien: Fug oder unfug.
\newblock {\em Im Internet unter: http://www. uni-konstanz.
  de/FuF/wiwi/laufer/lecture2/kriterien-text. html}, 2008.

\bibitem{mas1995microeconomic}
A.~Mas-Colell, M.~Dennis Whinston, J.~R Green, et~al.
\newblock {\em Microeconomic theory}, volume~1.
\newblock Oxford university press New York, 1995.

\bibitem{mastrolia2014malliavin}
T.~Mastrolia, D.~Possama{\"\i}, and A.~R{\'e}veillac.
\newblock On the {M}alliavin differentiability of {BSDE}s.
\newblock {\em Annales de l'institut Henri Poincar{\'e}, Probabilit{\'e}s et
  Statistiques $(${\rm B}$)$}, to appear, 2014.

\bibitem{miettinen2012nonlinear}
K.~Miettinen.
\newblock {\em Nonlinear multiobjective optimization}, volume~12.
\newblock Springer Science \& Business Media, 2012.

\bibitem{mirrlees1974notes}
J.A.~Mirrlees.
\newblock Notes on welfare economics, information and uncertainty.
\newblock In M.S. Balch, D.L. McFadden, and S.Y. Wu, editors, {\em Essays on
  economic behavior under uncertainty}, pages 243--261. Amsterdam: North
  Holland, 1974.

\bibitem{mirrlees1976optimal}
J.A.~Mirrlees.
\newblock The optimal structure of incentives and authority within an
  organization.
\newblock {\em The Bell Journal of Economics}, 7(1):105--131, 1976.

\bibitem{panaccione2007pareto}
L.~Panaccione et~al.
\newblock Pareto optima and competitive equilibria with moral hazard and
  financial markets.
\newblock {\em The BE Journal of Theoretical Economics}, 7(1):1--21, 2007.

\bibitem{pardoux1990adapted}
\'E.~Pardoux and S.~Peng.
\newblock Adapted solution of a backward stochastic differential equation.
\newblock {\em Systems \& Control Letters}, 14(1):55--61, 1990.

\bibitem{pardoux1992backward}
\'E.~Pardoux and S.~Peng.
\newblock Backward stochastic differential equations and quasilinear parabolic
  partial differential equations.
\newblock In B.L. Rozovskii and R.B. Sowers, editors, {\em Stochastic partial
  differential equations and their applications. Proceedings of IFIP WG 7/1
  international conference University of North Carolina at Charlotte, NC June
  6--8, 1991}, volume 176 of {\em Lecture notes in control and information
  sciences}, pages 200--217. Springer, 1992.

\bibitem{peet1998maastricht}
J.~Peet.
\newblock Maastricht follies.
\newblock {\em Economist}, 347(8063), 1998.

\bibitem{possamai2015stochastic}
D.~Possama{\"\i}, X.~Tan, and C.~Zhou.
\newblock Stochastic control for a class of nonlinear kernels and applications.
\newblock {\em arXiv preprint arXiv:1510.08439}, 2015.

\bibitem{prescott1984pareto}
E.~C Prescott and R.~M Townsend.
\newblock Pareto optima and competitive equilibria with adverse selection and
  moral hazard.
\newblock {\em Econometrica: Journal of the Econometric Society}, pages 21--45,
  1984.

\bibitem{rouge2000pricing}
R.~Rouge and N.~El~Karoui.
\newblock Pricing via utility maximization and entropy.
\newblock {\em Mathematical Finance}, 10(2):259--276, 2000.

\bibitem{sannikov2008continuous}
Y.~Sannikov.
\newblock A continuous--time version of the principal--agent problem.
\newblock {\em The Review of Economic Studies}, 75(3):957--984, 2008.

\bibitem{scheller2004european}
H~K Scheller.
\newblock The european central bank.
\newblock {\em History, Role and Functions. ECB}, 2004.

\bibitem{soner2012wellposedness}
H.~Mete Soner, N.~Touzi, and J.~Zhang.
\newblock Wellposedness of second order backward sdes.
\newblock {\em Probability Theory and Related Fields}, 153(1-2):149--190, 2012.

\bibitem{sung2001lectures}
J.~Sung.
\newblock Lectures on the theory of contracts in corporate finance: from
  discrete--time to continuous--time models.
\newblock {\em Com2Mac Lecture Note Series}, 4, 2001.

\bibitem{touzi2012optimal}
N.~Touzi.
\newblock {\em Optimal stochastic control, stochastic target problems, and
  backward SDE}, volume~29.
\newblock Springer Science \&amp; Business Media, 2012.

\bibitem{walras1896elements}
L.~Walras.
\newblock {\em {\'E}l{\'e}ments d'{\'e}conomie politique pure; ou, Th{\'e}orie
  de la richesse sociale}.
\newblock F. Rouge, 1896.

\bibitem{xing2016class}
H.~Xing and G.~{\v{Z}}itkovi{\'c}.
\newblock A class of globally solvable {M}arkovian quadratic {BSDE} systems and
  applications.
\newblock {\em arXiv preprint arXiv:1603.00217}, 2016.

\end{thebibliography}
\end{document}